\documentclass[reqno]{amsart}
\usepackage[mathscr]{eucal}
\usepackage{amscd}
\usepackage{color}
\usepackage{amsfonts}
\usepackage{amsmath, amssymb}
\usepackage{mathtools}
\usepackage{latexsym}
\usepackage{ulem} 

\usepackage{tikz}
\usetikzlibrary{patterns}
\usepackage{bm}

\numberwithin{equation}{section}

\newtheorem{theorem}{Theorem}[section]
\newtheorem{thm}{Theorem}[section]

\newtheorem{definition}[theorem]{Definition}

\newtheorem{proposition}[theorem]{Proposition}
\newtheorem{remark}[theorem]{Remark}

 \DeclareMathAlphabet\mathcal{OMS}{cmsy}{m}{n}
\SetMathAlphabet\mathcal{bold}{OMS}{cmsy}{b}{n}

\newcommand{\mR}{\mathbb{R}}

\newcommand{\mN}{\mathbb{N}}

\newcommand{\um}[1]{$\ddot{\text{#1}}$}

\newcommand{\Om}{\Omega}
\newcommand{\om}{\omega}

\newcommand{\cd}{\mathcal{D}}
\newcommand{\ocd}{\overline{\mathcal{D}}}

\newcommand{\cf}{\mathcal{F}}
\newcommand{\ch}{\mathcal{H}}
\newcommand{\cg}{\mathcal{G}}

\begin{document}

\title[Malliavin Calculus for the one-dimensional stochastic Stefan problem]
{Malliavin Calculus for the one-dimensional stochastic Stefan problem
}

\author[Antonopoulou, Dimitriou, Karali, Tzirakis]{D. C. Antonopoulou$^{\ddag}$, D. Dimitriou$^{*}$, G. Karali$^{\ddag *}$, K. Tzirakis$^{\dag *}$}

\thanks
{$^{\dag}$ Department of Mathematics and Applied Mathematics,
University of Crete, GR--714 09 Heraklion, Crete, Greece.} 
\thanks
{$^{\ddag}$ Department of Mathematics, National and Kapodistrian
University of Athens, Panepistimiopolis, 15784, Athens,
Greece.}
\thanks
{$^{*}$ Institute of Applied and Computational Mathematics,
FORTH, GR--711 10 Heraklion, Greece.}
\thanks{E-mails: danton@math.uoa.gr, dimitriou.dimos@gmail.com,
gkarali@math.uoa.gr, kostas.tzirakis@gmail.com}

%
%
\begin{abstract}
We consider the one-dimensional outer stochastic Stefan problem with reflection. 
The problem admits maximal solutions as long as the velocity of the moving boundary remains bounded, \cite{AFK2,BH,BH2}. We apply Malliavin calculus to the transformed equation and first prove that its maximal solution $u$ has continuous paths a.s. In the case of the unreflected problem, the previous enables the localization of a proper approximating sequence of the maximal solution. Then, we derive there locally the differentiability of maximal $u$ in the Malliavin sense. The novelty of this work, apart from the derivation of continuity of the paths for the maximal solution with reflection, is that for the unreflected case we introduce a localization argument on maximal solutions and define efficiently the relevant sample space. More precisely, we prove the local (in the sample space) existence of the Malliavin derivative and, under a non-degeneracy condition on the noise coefficient, the absolute continuity of the law of the solution with respect to the Lebesgue measure.
\end{abstract}


\maketitle
\pagestyle{myheadings}
\thispagestyle{plain}
%
%

\section{Introduction}
\subsection{The Stochastic Stefan problem}
We consider two well separated phases $\Omega_{{\rm Sol}}$ and $\Omega_{{\rm Liq}}$, the solid and liquid phase respectively, occupying a static bounded domain $\Omega$ of $\mathbb{R}$ defined
by the interval $\Omega:=(a,b)$. Let $\Omega_{{\rm Sol}}(t):=(s^-(t),s^+(t))$ be the solid phase at time $t$, and
\begin{equation}\label{lam}
\Omega_{{\rm Liq}}(t) = \Omega\setminus [s^-(t), s^+(t)] ,
\end{equation}
so
$$\Omega=\Omega_{{\rm Liq}}(t)\cup \{s^-(t),s^+(t)\}\cup \Omega_{{\rm Sol}}(t) . 
$$
We assume that the initial solid phase is in $\Omega$ and far from its boundary, i.e., that $\lambda\gg s^+(0)-s^-(0)$, for $\lambda:=b-a$ the length of $\Omega$.

The outer parabolic stochastic Stefan problem is defined by
\begin{equation}\label{stef1}\left\{
\begin{aligned}
&\partial_t w=\alpha\Delta w +\sigma\big({\rm dist}(y,\partial
\Omega_{\rm{Sol}})\big) \dot{W}_s(y,t)+\dot{\eta}_s(y,t), \qquad y\in\Omega_{\rm{Liq}}(t),\quad t>0\quad \mbox{(`liquid' phase)},
\\
&w(y,t)=0,\qquad y\in \overline{\Omega_{\rm{Sol}}(t)},\quad t>0 , \qquad \mbox{(`solid' phase)},
\\
&V=-\nabla w|_{\partial\Omega_{\rm{Sol}}(t)},\quad  t>0 , \qquad\qquad\qquad \mbox{(Stefan condition)},
\\
&w(a,t)=w(b,t)=0,\quad t>0 ,
\\
&\partial \Omega_{\rm{Sol}}(0)=\{s^-(0),s^+(0)\}={\rm given}.
\end{aligned}
\right.
\end{equation}
Here, $w=w(y,t)$ is a density, $\alpha> 0$ is a positive constant, and $\sigma$ is a function of the distance of $y$ from the solid phase boundary, $ {\rm dist}(y,\partial\Omega_{\rm{Sol}}(t)):=\min\{|y-s^+(t)|,|y-s^-(t)|\}$. The noise term is given by
\begin{equation*}\label{noi1}
\dot{W}_s(y,t) := \dot{W}(y-s^+(t),t) \;\; \text{if}\;\;y\geq
s^+(t), \;\;
\dot{W}_s(y,t):=\dot{W}(-y+s^-(t),t) \;\; \text{if}\;\;y\leq s^-(t),
\end{equation*}
where $\dot{W}(\pm y\mp s^\pm(t),t)$ is a space-time white noise,
see for example in \cite{Walsh,DQS}. The initial condition $w(y,0)=w_0(y) $ is a sufficiently smooth deterministic function compactly supported in the liquid phase. The moving boundary of \eqref{stef1} is the union for all $t\geq 0$ of the curves $y=s^+(t),\;\; y=s^-(t)$ enclosing the solid phase with midpoint $s(t):=(s^-(t)+s^+(t))/2$ and length $s^+(t)-s^{-}(t)$ defining the spread. The solution $w(y,t)$ vanishes when $y$ is in the closure of the solid phase. The Stefan condition describes the
change of liquidity; $V$ there is the velocity of the interface, and for $(\nabla \cdot )^\pm$ denoting the derivative from the right ($y> s^+$) and left ($y<s^-$) it follows that
$$
V(s^+(t),t):=\partial_t s^+(t)=-(\nabla w)^+(s^+(t),t),\;\;\;\;
V(s^-(t),t):=\partial_t s^-(t)=-(\nabla w)^-(s^-(t),t).$$ 
The term on \eqref{stef1} is given by
$$\dot{\eta}_s(y,t) := \dot{\eta}_1(y-s^+(t),t)\quad \text{if}\quad y\geq s^+(t), \qquad
\dot{\eta}_s(y,t):=\dot{\eta}_2(-y+s^-(t),t) \quad \text{if}\quad y\leq s^-(t),$$
where $\eta_1$, $\eta_2$ are reflections.


\subsection{The transformed problem}
The liquid phase consists of two separate bounded linear segments which enables the splitting of the Stefan problem equation in two equations posed for on $y\geq s^+$ and on $y\leq s^-$. After the
change of variables
\begin{equation*}\label{cv1gen}
x:= y-s^+(t) \;\; \text{if }y\geq s^+(t), \qquad
x:=-y+s^-(t) \;\; \text{if }y\leq s^-(t),
\end{equation*}
cf. \cite{AFK2}, the stochastic equation of \eqref{stef1} is transformed by the use of the Stefan condition into two independent ones each posed on the fixed space domain $\mathcal{D}:=(0,\lambda)$ with Dirichlet boundary conditions. The value $ x=0 $ occurs when the price $y$ is $s^{\pm}$, while $x=\lambda$ when the spread is zero and $s^+=s^-$ hits the boundary of $\Omega$. These equations are of the form
\begin{equation}\label{meq11}
u_t(x,t)=\alpha  u_{xx}(x,t)\mp u_x(0^+,t)
u_x(x,t)\pm\sigma(x)\dot{W}(x,t)+\dot{\eta}(x,t),\;\;x\in\mathcal{D},\;\;t>0,
\end{equation}
with $\eta(dx,dt)$ a random measure on $ D\times \mathbb R_+ $ (reflection) keeping $u$ a.s.\ non-negative, and a space-time white noise $\dot{W}(x,t)$ where $W=W(x,t)$ is a Wiener process. We also note that when a system is considered in place of the Stefan problem \eqref{stef1} with different liquidity coefficients $\alpha_1$, $\alpha_2$ on the diffusion term $\Delta w$, a problem analyzed in \cite{BH} when the spread is zero, the same equation of the above general form \eqref{meq11} will appear after the change of variables for $\alpha=\alpha_1,\alpha_2$.

Without restricting the generality, we consider one of the cases for the s.p.d.e.\ which yields the following transformed initial and boundary value problem
\begin{equation}\label{trstef}
\left\{
\begin{aligned}
&u_t(x,t)= \alpha
u_{xx}(x,t)-u_x(0,t)u_x(x,t)+\sigma(x)\dot{W}(x,t)+\dot{\eta}(x,t),\quad
(x, t)\in\cd\times(0,T),\\
& u(x,0) =
u_0(x),\;\;\forall x\in\overline{\cd},\\
&u(0,t) = u(\lambda,t) = 0,\;\;\forall t\in[0,T].
\end{aligned}
\right.
\end{equation}
As in \cite{AFK2}, we assume that
\begin{equation}\label{sassum}
\sigma:\overline{\cd}\rightarrow\mR\;\;\mbox{ is a function
in}\;\;C(\overline{\cd}),\;\;\mbox{and differentiable
at}\; x=0,\;\; \mbox{ with }\;\sigma(0)=\sigma(\lambda)=0,
\end{equation}
and, since $w_0$ is compactly supported on the initial liquid
phase, that
\begin{equation}\label{u0assum}
u_0\;\;\mbox{is a deterministic and non-negative function in}\;\;
C^\infty_c(\overline{\mathcal{D}}).
\end{equation}
The reflection measure $\eta$ keeps $u$ non-negative and satisfies
\begin{equation}\label{psi}
\begin{split}
&\mbox{for all}\mbox{ measurable functions }
\psi:\;\overline{\mathcal{D}}\times (0,T)\rightarrow [0,\infty)\\
&\int_0^t\int_{\mathcal{D}}\psi(x,s)\eta(dx,ds)\;\;\mbox{is}\;\;\mathcal{F}_t-\mbox{measurable}
\end{split}
\end{equation}
for $\cf_{t}$ the filtration generated by $\{W(x,s): 0 \leq s\leq
t, x\in\cd\}$, and the constraint
\begin{equation}\label{refl}
\int_0^T\int_{\mathcal{D}} u(x,s)\eta(dx,ds)=0.
\end{equation}

The problem \eqref{trstef}, \eqref{psi}, with $\sigma, u_0 $ satisfying \eqref{sassum} and \eqref{u0assum} respectively, has a unique maximal solution $(u,\eta)$ in the time interval $[0,T^*)$, where
\begin{equation}\label{stime}
T^*:=\displaystyle{\sup_{M>0}}\tau_M,
\end{equation}
for
$$\tau_M:=\inf\Big{\{}T\geq
0:\;\displaystyle{\sup_{r\in[0,T)}}|u_x(0^+,r)|\geq M\Big{\}}
\equiv\inf\Big{\{}T\geq 0:\;\displaystyle{\sup_{r\in[0,T)}}
u_x(0^+,r)\geq M\Big{\}} ,
$$
see in \cite{AFK2}. An analogous result for the same stopping time $T^*:=\sup_{M>0}\tau_M$ was derived in \cite{BH} for a system of such SPDEs when $\alpha=\alpha_1,\alpha_2$.

The solution $ u $ of problem \eqref{trstef} is written in integral representation
$\forall x\in \cd,\ t\in[0,T^{*})$ as
\begin{equation}\label{int}
\begin{split}
u(x,t) = & \int_{\mathcal{D}}u_0(y) \,G(x,y,t)\, dy +\int_0^t\int_{\mathcal{D}}u_y(0,s) \,G_{y}(x,y,t-s) \,u(y,s)\, dy\,ds\\
&+\int_0^t\int_{\cd}G(x,y,t-s)\sigma(y) \, W(dy,ds)
+\int_0^t\int_{\cd} G(x,y,t-s) \,\eta(dy,ds),
\end{split}
\end{equation}
where $G$ is the Green's function of the Dirichlet problem on
$\overline{\cd}$ of the heat equation $v_t=\alpha v_{xx}$. 

We recall the Banach space
$(\mathcal{H},\|\cdot\|_{\mathcal{H}})$, introduced in \cite{BH},
where
\begin{equation}\label{bansp}
\ch:=\big{\{}f\in C(\overline{\mathcal{D}}):\; \exists\ f'(0),\;\;
\mbox{and}\;\;\ f(0)=f(\lambda)=0\big{\}},\;\;\;
\text{and}\;\;\;\;\|f\|_{\ch}:=\displaystyle{\sup_{x\in\mathcal{D}}}\bigg{|}\dfrac{f(x)}{x}\bigg{|}.
\end{equation}

Let us consider $M>0$ deterministic and $T>0$ deterministic, and for a positive integer $n$ define the stopping time
\begin{equation}\label{stimen}
\tilde{\tau}_n:=\inf\Big{\{}T\geq
0:\;\displaystyle{\sup_{r\in[0,T)}}\|u(\cdot,r)\|_{\mathcal{H}}=\displaystyle{\sup_{r\in[0,T)}}\, \displaystyle{\sup_{x\in\mathcal{D}}}\, \bigg{|}\dfrac{u(x,r)}{x}\bigg{|}\geq
n\Big{\}},
\end{equation}
it then holds that, cf.
\cite{BH},
\begin{equation}\label{imp1}
E\Big(\sup_{t\in[0,\min\{T,\tau_M,\sup_{n>0}\tilde{\tau}_n\})}\|u(\cdot,t)\|_{\mathcal{H}}^p\Big)\leq
C(T,M,p)<\infty,
\end{equation}
where we used that in our problem
$h(u_y(0,s))=u_y(0,s)$ which has a linear growth as in the
result of \cite{BH}.
\begin{remark}\label{glob}
Note that the bound of the above expectation depends on the choice of $M$. In the case where the term $u_{y}(0,s)$ was replaced by some $h(u_{y}(0,s))$ uniformly bounded for any $M>0$ then the
solution $u$ would be global, cf. Corollary 5.7 of \cite{BH}, and it would hold that
\begin{equation}\label{imp2}
E\Big{(}\displaystyle{\sup_{t\in[0,\min\{T,\sup_{M>0}\tau_M,\sup_{n>0}\tilde{\tau}_n\})}}\|u(\cdot,t)\|_{\mathcal{H}}^p\Big{)}\leq
C(T,p)<\infty.
\end{equation}
\end{remark}

We fix now $T>0$ deterministic and $M>0$ deterministic, and define
\begin{equation}\label{omm}
\Omega_M:=\Big\{\omega\in\Omega:\;\displaystyle{\sup_{t\in[0,\min\{T,\tau_M\})}}
|u_x(0^+,t,\omega)|< M\Big\},
\end{equation}
and for any integer $n>0$, we define
\begin{equation}\label{ommn}
\Omega_M^n:=\Big\{\omega\in\Omega_M:\;\displaystyle{\sup_{t\in[0,\min\{T,\tau_M\})}}
\|u(\cdot,t,\omega)\|_{\mathcal{H}}<n\Big\}.
\end{equation}
It then follows that
$$\Omega_M^1\subseteq\Omega_M^2\subseteq\cdots\subseteq
\Omega_M,$$
and
$ \Omega_M^n\uparrow\Omega_M$ as
$n\rightarrow\infty$.
\begin{remark}\label{nonemp}
As we assume that the initial condition $u_0$ satisfies \eqref{u0assum}, then it follows that for $t=0$ $u(x,0)=u_0(x)$ satisfies the bounds in the definitions of $\Omega_M$, and of $ \Omega_M^n $ for $M$ large enough and for all $n\geq n_0 $ respectively. This yields that the sets $\Omega_M$ and $\Omega_M^n$ are not empty for $M$ large enough such that $\nabla u_0(0)< M$ and for all $n\geq n_0$ for $n_0$ large enough such that $\|u_0(\cdot)\|_{\mathcal{H}}^p<n_0$, where of course $M$, $n_0$ depend on the initial condition $u_0$.
\end{remark}

\subsection{Green's function} Throughout this paper, we denote by $G$ the Green's function for the heat equation $u_t(x,t)= \alpha u_{xx}(x,t)$ on $[0,\lambda] $ with Dirichlet boundary conditions, which is given by
\begin{equation}\label{kostasgreen1}
G(x,y,t-s) = \frac{2}{\lambda} \sum_{k=1}^\infty \sin(\tfrac{k\pi}{\lambda} y)\, \sin(\tfrac{k\pi}{\lambda} x)\, e^{-\alpha (\pi/\lambda)^2 k^2 (t-s)}  ,
\end{equation}
 for $s<t. $ An equivalent expression is
\begin{equation}\label{kostasgreen2}
G(x,y,t-s)=\frac{1}{\sqrt{4\pi\alpha(t-s)}}\sum_{k=-\infty}^\infty\left[\exp\Big(-\frac{(y-x+2k\lambda)^2}{4\alpha(t-s)}\Big)-\exp\Big(-\frac{(x+y+2k\lambda)^2}{4\alpha(t-s)}\Big)\right].
\end{equation}
We will use the following estimates of $ G $ and its derivative $G_y $ (see \cite{EidelIvas}, \cite{Solonnikov}). We have, for some constants $c,C>0, $
$$ |G(x,y,t-s)|\leq C\,|t-s|^{-1/2} \, \exp\Big(\!-c\,\frac{|x-y|^2}{|t-s|}\Big), $$
and so
\begin{equation}\label{ggg}
|G(x,y,t-s)|^p\leq
 C\,|t-s|^{-\frac{p}{2}}\,\exp\Big(\!-c\,\frac{|x-y|^{2}}{|t-s|}\Big).
\end{equation}
We also have
\begin{equation}\label{gen}
|G_y(x,y,t-s)|\leq
c|t-s|^{-1}\exp\Big{(}\!-c\,\frac{|x-y|^2}{|t-s|}\Big{)},
\end{equation}
and so
\begin{equation}\label{gy}
\int_{\mathcal{D}}|G_y(x,y,t-s)|dy\leq
c|t-s|^{-1}\int_{\mathcal{D}}\exp\Big(\!-c\frac{|x-y|^2}{|t-s|}\Big)dy\leq
c|t-s|^{-1/2} ,
\end{equation}
where we used the Gaussian integral
\begin{equation}\label{kostasgaus}
\int_{\mathbb{R}}\exp\Big{(}-c\,\frac{|x-y|^{2}}{|t-s|}\Big{)}dy = \sqrt{\tfrac{\pi}{c}}\; |t-s|^{\frac{1}{2}}.
\end{equation}

\subsection{Main Results}
In Section 2, we consider the problem with reflection and prove that the paths of the sample space $\Omega_M$, where $M>0$ is an arbitrary deterministic constant, are space-time continuous, i.e., for all realizations in $\Omega_M$ where as proven in \cite{BH,AFK2} the solution exists, it is also space-time continuous. Section 3 is devoted to the localization of $u$ in the given sample space $\Omega_M$ when the unreflected problem is considered. We present there some important definitions of the spaces where we investigate the Malliavin differentiability. We then define a proper approximation $u_n$ of $u$ for which $u_n=u$ on $\Omega_M^n$ and prove that $u_n$ exists uniquely as a solution of a cut-off spde. Finally, in Section 4, we consider $M_{\rm d}>0$ deterministic and arbitrary. We prove that the Malliavin derivative of $u$ exists for times up to which the Malliavin derivative of the term in the definition of $\Omega_M$ (i.e., $u_x(0^+,t,\omega)$) not only is bounded by $M$ but also stays differentiable in the Malliavin sense with derivative upper bounded by $M_{\rm d}$. We additionally prove some estimates for the derivative and establish the absolute continuity of the law of the solution with respect to the Lebesgue measure on $\mathbb R.$

\section{Space-Time Continuity for the paths of $\Omega_M$}
In this section we prove the next very useful theorem establishing
the space-time continuity of $u(\cdot,\cdot,\omega)$ for any
$\omega\in\Omega_M$, which is essential for the localization
argument.

\begin{thm}\label{thmcon}
Let $u_0$, $\sigma$ satisfy \eqref{u0assum} and \eqref{sassum}
respectively and $M>0$, $T>0$ deterministic fixed values, and $M$
large enough such that $\nabla u_0(0)<M$. Consider
$\omega\in\Omega_M$ for $\Omega_M$ defined by \eqref{omm}. The
unique maximal solution $u(x,t,\omega)$ of \eqref{int} is
continuous in space-time for any $x\in\overline{\mathcal{D}}$ and
any $t\in[0,\min\{T,\tau_M\})$ for $\tau_M$ the stopping
time defined in \eqref{stime}.
\end{thm}
\begin{proof}
Since $u(\cdot,t,\omega)\in\mathcal{H}$ for any fixed $t\in[0,\tau_M)$, cf.\ \cite{AFK2}, then by the definition of $\mathcal{H}$, $u(\cdot,t,\omega)\in C(\overline{\mathcal{D}}), $ i.e., $u$ is space-continuous for any fixed $t\in[0,\tau_M)$, where note that $u_0(x)=u(x,0)\in C_c^\infty(\overline{\mathcal{D}})\subset
\mathcal{H}$.\\
Let $t_1, t_2 \in [0,\min\{T,\tau_M\})$ with $t_2\geq t_1$.
In order to ease notation we will use $u(x,t)$ in place of
$u(x,t,\omega)$.
\allowdisplaybreaks
By the integral representation \eqref{int} of $u$, we get
\begin{eqnarray*}
u(x,t_2)-u(x,t_1)
&= &\int_{\cd}u_0(y)\big(G(x,y,t_2)-G(x,y,t_1)\big)\,dy\\
& &+\int_0^{t_2}\int_{\mathcal{D}}u_y(0,s)u(y,s)G_y(x,y,t_2-s)\,dy\,ds\\
& &-\int_0^{t_1}\int_{\mathcal{D}}u_y(0,s)u(y,s)G_y(x,y,t_1-s)\,dy\,ds\\
& & +\int_0^{t_2}\int_{\cd}\sigma(y)G(x,y,t_2-s)W(dy,ds)-\int_0^{t_1}\int_{\cd}\sigma(y)G(x,y,t_1-s)W(dy,ds)\\
& & +\int_0^{t_2}\int_{\cd}G(x,y,t_2-s)\,\eta(dy,ds)-\int_0^{t_1}\int_{\cd}G(x,y,t_1-s)\,\eta(dy,ds)\\
&=&\int_{\cd}u_0(y)\big(G(x,y,t_2)-G(x,y,t_1)\big)dy\\
& & +\int_{t_1}^{t_2}\int_{\mathcal{D}}u_y(0,s)u(y,s)G_y(x,y,t_2-s)\,dy\,ds\\
& & +\int_0^{t_1}\int_{\mathcal{D}}u_y(0,s)u(y,s)\bigl(G_y(x,y,t_2-s)-G_y(x,y,t_1-s)\bigr)\,dy\,ds\\
& & + A(x,t_2) - A(x,t_1),
\end{eqnarray*}
for
$$A(x,t):=\int_0^t\int_{\cd}G(x,y,t-s)\,\sigma(y)\,W(dy,ds)\,+\,\int_0^t\int_{\cd}G(x,y,t-s)\,\eta(dy,ds).$$

\noindent
So, by using \eqref{imp1}, we have 

\begingroup
\allowdisplaybreaks
\begin{eqnarray}
|u(x,t_2)-u(x,t_1)| & \leq & \Big{|}\int_{\cd}u_0(y)\,(G(x,y,t_2)-G(x,y,t_1))\, dy\Big{|}\nonumber\\
& & +\int_{t_1}^{t_2}\int_{\mathcal{D}}u_y(0,s)\,u(y,s)\, \big{|}G_{y}(x,y,t_2-s)\big{|}\,dy\,ds\nonumber\\
& & +\int_0^{t_1}\int_{\mathcal{D}}u_y(0,s)\,u(y,s)\, \big{|}G_y(x,y,t_2-s)-G_y(x,y,t_1-s)\big{|}\,dy\,ds\nonumber\\
& & +\big{|}A(x,t_2)-A(x,t_1)\big{|}\nonumber\\
& \leq & \|u_0\|_{L^{\infty}(\cd)}\int_{\cd} \big{|}G(x,y,t_2)-G(x,y,t_1)\big{|}\,dy\nonumber\\
& & +\int_{t_1}^{t_2}\int_{\mathcal{D}}|u_y(0,s)||y|\Big{|}\frac{u(y,s)}{y}\Big{|}\ \big{|}G_y(x,y,t_2-s)\big{|}\,dy\,ds\nonumber\\
& & +\int_0^{t_1}\int_{\mathcal{D}}|u_y(0,s)||y|\Big{|}\frac{u(y,s)}{y}\Big{|}\ \big{|}G_y(x,y,t_2-s)-G_y(x,y,t_{1}-s)\big{|}\,dy\,ds\nonumber\\
& & +\big{|}A(x,t_2)-A(x,t_1)\big{|}\nonumber\\
& \leq & \|u_0\|_{L^{\infty}(\cd)}\int_{\cd} \big{|}G(x,y,t_2)-G(x,y,t_1)\big{|}\,dy\nonumber\\
& &+C(M,T,\omega)\int_{t_1}^{t_2}\int_{\mathcal{D}}\big{|}G_y(x,y,t_2-s)\big{|}\,dy\,ds\nonumber\\
& &+C(M,T,\omega)\int_0^{t_1}\int_{\mathcal{D}}\big{|}G_{y}(x,y,t_2-s)-G_y(x,y,t_1-s)\big{|}\,dy\,ds\nonumber\\
& &+\big{|}A(x,t_2)-A(x,t_1)\big{|}.\label{keq5}
\end{eqnarray}
\endgroup

\noindent
For the first term in \eqref{keq5}, observe first that the solution  $v(x,t)$ of the parabolic problem

$$v_t=\alpha v_{xx}\;\;\forall(x,t)\in\mathcal{D}\times
(0,T),\;\;v(x,0)=u_0(x)\;\;\forall x\in\mathcal{D},\;\;v(0,t)=v(\lambda,t)=0\;\;\forall t\in(0,T),$$
is given by
$$v(x,t) = \int_{\cd}u_0(y)\,G(x,y,t)\,dy,$$
and so 

$$\Big{|}\int_{\cd}u_0(y) (G(x,y,t_2)-G(x,y,t_1))\,
dy\Big{|}=|v(x,t_2)-v(x,t_1)|\leq
\displaystyle{\sup_{(x,t)\in\mathcal{D}\times
(0,T)}}|v_t(x,t)||t_2-t_1|.$$ Using the equation that $v$
satisfies, we have

$$v_t=\alpha v_{xx},\;\;v_{tt}=\alpha
v_{xxt}=\alpha^2v_{xxxx},\;\;v_{ttt}=\alpha^2v_{xxxxt}=\alpha^3v_{xxxxxx},$$
$$v_{tx}=\alpha v_{xxx},\;\;v_{txx}=\alpha v_{xxxx},\;\;v_{ttx}=\alpha^2v_{xxxxx},$$
and so, we obtain that
$$\displaystyle{\sup_{(x,t)\in\mathcal{D}\times
(0,T)}}|v_t(x,t)|\leq c\|v_t\|_{H^2(\mathcal{D}\times (0,T))}\leq
c\displaystyle{\sup_{t\in (0,T)}}\|v\|_{H^6(\mathcal{D})},$$
where we used that in two dimensions (one for time, one for space) the $L^\infty$ norm is upper bounded by the $H^2$ norm. 
By parabolic regularity, if $u_0\in H^6(\mathcal{D})$, then $v\in L^\infty(0,T;H^6(\mathcal{D}))$.

Of course as we assume a stronger initial value regularity, $u_0\in C_c^{\infty}(\overline{\mathcal{D}})$, we may get directly, see \cite{Evans} Theorem 7 of \S7.1 at pg. 367, that $$\Big{|}\int_{\cd}u_0(y)\big(G(x,y,t_2)-G(x,y,t_1)\big)
dy\Big{|}\leq c|t_2-t_1|,$$ uniformly for all $(x,t)\in\mathcal{D}\times (0,T)$, where $c=c(\omega)$ if $u_0$ was stochastic. In our case we considered $u_0$ deterministic and so $c$ is not depending on the realization.

For the second term in \eqref{keq5}, using \eqref{gy}, we get
\begin{align*}
\int_{t_1}^{t_2}\int_{\mathcal{D}}\big{|}G_y(x,y,t_2-s)\big{|}\, dy\,ds\ \leq C\, |t_2-t_1|^{1/2}\rightarrow 0 ,
\end{align*}
almost surely in $\Omega_M$ as $ t_2\rightarrow t_1$.

For the third term, according to \cite[Proposition A.6]{BH2}, there exists a constant $ C>0$ such that, for
$p>4$,
\begin{align*}
\biggl{(}\int_0^{t_1}\biggl{(}\int_{\mathcal{D}}\big{|}G_y(x,y,t_2-s)-G_y(x,y,t_1-s)\big{|}\ dy\biggr{)}^{\frac{p}{p-2}}ds\biggr{)}^{\frac{p-2}{p}}\leq C\, |t_2-t_1|^{\frac{p-4}{4p}} , \quad \forall x, y \in[0,\lambda] ,
\end{align*}
thus, using H\"older's inequality, we get
$$\int_0^{t_1}\int_{\mathcal{D}}\big{|}G_y(x,y,t_2-s) -G_y(x,y,t_1-s)\big{|}\,dy\,ds\, \leq \, (\lambda t_1)^{\frac{2}{p}}\ |t_2-t_1|^{\frac{p-4}{4p}}\rightarrow 0 , $$
almost surely in $\Omega_M$ as $t_2\rightarrow t_{1}$.

Lastly, regarding the process $A(x,t)$, according to Proposition
3.11 in \cite{BH2}, the first term is almost surely continuous as a function of
$(x,t)$ whereas the second term, according to \cite{AFK2}, coincides with the solution of a specific
Heat equation Obstacle problem which, by Theorem 3.2 of \cite{BH}, admits a H\um{o}lder continuous solution both in space and time. Thus $A(x,t)$ is almost surely continuous in space and time
and so
$$|A(x,t_2)-A(x,t_1)|\rightarrow0\ \text{almost surely in }\Omega_M\;\;\mbox{as}\  t_2\rightarrow t_1.$$
From all the above arguments, we conclude that
$$|u(x,t_2)-u(x,t_1)|\rightarrow 0\ \text{almost surely in }\Omega_M\;\;\mbox{as}\
t_2\rightarrow t_1 , $$ which proves the almost sure in
$\Omega_M$ continuity of $u$ in time.
\end{proof}

\section{Localization of the unreflected $u$ in $\Omega_M$} Let $(\tilde{\Omega},\mathcal{F},P)$ be a probability space and $X:\tilde{\Omega}\rightarrow\mR$ a random variable, and recall that the sample space $\tilde{\Omega}$ consists of  all the possible outcomes $\om\in\tilde{\Omega}$ (simple events) of a random experiment. The Malliavin derivative measures the rate of change of $X$ as a function of $\om\in\tilde{\Omega}$.

\allowdisplaybreaks

\subsection{Basic definitions}
\begin{definition}
The set $L^{2}\bigl{(}\tilde{\Omega}\times\cd\times[0,T]\bigr{)}$ consists of the stochastic processes $v(\om;x,t),\ (x,t)\in\cd\times[0,T]$ such that
$$\|v\|_{L^2(\tilde{\Omega}\times\mathcal{D}\times[0,T])} := \left(E\biggl{(}\int_0^T\int_{\mathcal{D}}|v(x,t)|^2\,dx\,dt\biggr{)}\right)^{1/2}<+\infty.$$
\end{definition}

\vspace{0.05cm}

\begin{definition}
We denote by $D^{1,2}\bigl{(}\cd\times[0,T]\bigr{)}$ the set of stochastic processes $v(\om;x,t),\ (x,t)\in\cd\times[0,T]$ such that the Malliavin derivative $($in space and time$)$, $D_{y,s}v(x,t)$, exists for any $y\in\cd$ and $s\geq0$ and any $(x,t)\in\cd\times[0,T]$, and satisfies $$\|v\|_{D^{1,2}(\cd\times[0,T])}:=\left(\|v\|^2_{L^2(\tilde{\Om})}+\|D_{.,.}v\|^{2}_{L^2(\tilde{\Omega}\times\cd\times[0,T])}\right)^{1/2}<+\infty.$$
\end{definition}

\begin{remark}
Recall that  $D^{1,2}$ is a Hilbert space.
\end{remark}
\vspace{0.05cm}
\begin{definition}
The set $L^{1,2}\bigl{(}\cd\times[0,T]\bigr{)}$ is defined as the class of all stochastic processes $v\in D^{1,2}\bigl{(}\cd\times[0,T]\bigr{)}$ that satisfy
$$\|v\|_{L^{1,2}(\cd\times[0,T])}:=\left(\|v\|^2_{L^2(\tilde{\Omega}\times\cd\times[0,T])}+\|Dv\|^2_{L^2(\tilde{\Omega}
\times(\cd\times[0,T])^{2})}\right)^{1/2}<+\infty \,, $$
where
\begin{align*}
\|Dv\|_{L^2(\tilde{\Omega}\times(\cd\times[0,T])^2)}:=
\left(E\left(\int_0^T\int_{\mathcal{D}}\int_0^T\int_{\cd}\big{|}D_{y,s}v(x,t)\big{|}^2\,dydsdxdt\right)\right)^{1/2}.
    \end{align*}
\end{definition}

\vspace{1cm}

We set $\tilde{\Omega}:=\Omega_M$ defined by \eqref{omm}.
\begin{definition}
The local versions of the sets $D^{1,2}\bigl{(}\cd\times[0,T]\bigr{)}$ and $L^{1,2}\bigl{(}\cd\times[0,T]\bigr{)}$ are defined as follows
$$L^{1,2}_{\text{loc}}(\cd\times[0,T]):=\{ \mbox{stoc. proc. } v:\; \exists\ \{(\Om_M^n,v_n)\}_{n\in\mN}\subset\mathcal{F}\times L^{1,2}
\ \text{such that}\ \Om_M^{n}\uparrow \Om_M\ \text{a.s.}\ \&\ v=v_{n}\ \mbox{a.s.\ on}\; \Om_M^n\},$$
$$D^{1,2}_{\text{loc}}(\cd\times[0,T]):=\{ \mbox{stoc. proc. } v:\; \exists\ \{(\Om_M^n,v_n)\}_{n\in\mN}\subset\mathcal{F}
\times D^{1,2}\ \text{such that}\ \Om_M^n\uparrow \Om_M\ \text{a.s.}\ \&\ v=v_n\ \mbox{a.s.\ on}\;
 \Om_M^n\},$$
where recall that

$$\Om_M^n\uparrow \Om_M\ \text{a.s.}\iff\Om_M^{1}\subseteq\Om_M^2\subseteq\dots\subseteq\Om_M
\quad \text{such that}\quad
\lim_{n\rightarrow+\infty} P(\Om_M^n) = P(\Om_M) = 1.$$
\end{definition}

\vspace{0.5cm}

Recall now that it holds that $L^{1,2}\subseteq
L^{1,2}_{\text{loc}}\subseteq D^{1,2}_{\text{loc}}$\ \text{as
well as}\ $L^{1,2}\subseteq D^{1,2}\subseteq
D^{1,2}_{\text{loc}}\subseteq D^{1,1}_{\text{loc}}$. Also recall
that if $v\in D^{1,2}_{\text{loc}}$ and $(\Om_M^n,v_n)$
localizes $v$ in $D^{1,2}$, then the Malliavin derivative $D_{y,s}v$
is defined without ambiguity by $D_{y,s}v=D_{y,s}v_n$ on $\Om_M^n,\ \forall n\in\mN$.

Since $L^{1,2}\subseteq D^{1,2}$, a constructed localization $(\Om_M^{n},v_{n})$ of $v$ in $L^{1,2}$ is also a localization in $D^{1,2}$ and thus, will define well the Malliavin derivative of $v$ through the Malliavin derivative of $v_n$.

\vspace{0.5cm}

Let us consider the unreflected problem, i.e., we set $\eta=0$ in
\eqref{int}. Here, we note that the bound
\eqref{imp1}, holds
also true, i.e.,
\begin{equation*}
E\Big{(}\displaystyle{\sup_{t\in[0,\min\{T,\tau_M,\sup_{n>0}\tilde{\tau}_n\})}}\|u(\cdot,t)\|_{\mathcal{H}}^p\Big{)}\leq
C(T,M,p)<\infty,
\end{equation*}
but we need to define for $\tau_M:=\inf\big{\{}T\geq
0: \sup_{r\in[0,T)}|u_x(0^+,r)|\geq M\big{\}}$
(which is here $\neq\inf\{T\geq
0:\;{\sup\limits_{r\in[0,T)}} u_x(0^+,r)\geq M\}$),
$\Omega_M$ as follows
\begin{equation}\label{ommnew}
\Omega_M:=\Big\{\omega\in\Omega:\;\displaystyle{\sup_{t\in[0,\min\{T,\tau_M\})}}
|u_x(0^+,t,\omega)|< M\Big\}.
\end{equation}
The absolute value over $u_x$ is introduced in the definition as done in places in \cite{AFK2} wherein the unreflected problem was also analyzed (since $u_x(0^+,t,\omega)$ may be negative in the unreflected case where $u$ changes sign in general).

\subsection{The cut-off solution $\mathbf{u_n=u}$ a.s.\ on $ \mathbf{\Omega_M^n}$}
For fixed $n\in\mN$, we define as in \cite{web,JDE18} the cut-off function
$H_n:\mathbb R\rightarrow\mathbb{R}^+$, $H_n$ smooth (in $ C^1(\mathbb R)$) such that for all $n,\; 0 \leq H_n\leq 1$ and $ |H'_n|\leq 2$ with
$$H_n(v)=
\begin{cases}
1\,, & \text{if }|v|<n, \\
0\,, & \text{if }|v|>n+1.
\end{cases}$$
We now define $T_n(v):=H_n(v)v$.

%

Let $T>0$ deterministic and $M>0$ deterministic. We recall the
definition of the stopping time $\tau_M$ (given in \eqref{stime}),
$\tau_M:=\inf\Big{\{}T\geq 0:\;\displaystyle{\sup_{r\in[0,T)}}
|u_x(0^+,r)|\geq M\Big{\}}$.

We define $\forall x\in\cd,\ t\in[0,\min\{T,\tau_M\})$, the cut-off stochastic integral equation
\begin{equation}\label{coint}
\begin{split}
u_n(x,t,\omega) =& \int_{\mathcal D} u_0(y) \, G(x,y,t) \, dy + \int_0^t\int_{\mathcal D} u_y(0,s) \, G_y(x,y,t-s) \,y\, T_n \bigl(y^{-1}u_n(y,s,\omega)\bigr)\, dy\,ds
\\[0.2cm]
&+\int_0^t\int_{\cd}G(x,y,t-s) \, \sigma(y) \, W(dy,ds),
\end{split}
\end{equation}
for $\omega\in\Omega_M$. 

Subtracting \eqref{int} and \eqref{coint}, then using \eqref{gy} and applying H\"older's inequality with exponents $p>2$ and $ p'=p/(p-1)$, we obtain, for $ \omega\in\Omega_M^n, $
\begin{eqnarray*}
\left|u_n(x,t)-u(x,t)\right|&\leq&
\int_0^t\int_{\mathcal D} |u_y(0,s)| \, |G_y(x,y,t-s)| \, \left|y\, T_n \bigl(y^{-1}u_n(y,s)\bigr)-u(y,s)\right| dy\,ds
\\
&\leq&
M  \int_0^t\int_{\mathcal D} |G_y(x,y,t-s)| \, \left|y\, T_n \bigl(y^{-1}u_n(y,s)\bigr)-u(y,s)\right| dy\,ds
\\
&\leq&
C_n\int_0^t \left(\int_{\mathcal D} |G_y(x,y,t-s)| \, dy\right)\, \left\|u_n(\cdot,s)-u(\cdot,s)\right\|_{L^\infty(\mathcal D)}\,ds
\\
&\leq&
c_n \int_0^t |t-s|^{-1/2} \left\|u_n(\cdot,s)-u(\cdot,s)\right\|_{L^\infty(\mathcal D)}\,ds
\\
&\leq&
c_n 
\left(\int_0^t |t-s|^{-p'/2}\,ds\right)^{1/p'}
\left(\int_0^t \left\|u_n(\cdot,s)-u(\cdot,s)\right\|^p_{L^\infty(\mathcal D)}\,ds\right)^{1/p}
\\
&\leq&
c_{n,p}\,T^{(p-2)/(2p)}\left(\int_0^t \left\|u_n(\cdot,s)-u(\cdot,s)\right\|^p_{L^\infty(\mathcal D)}\,ds\right)^{1/p}\,,
\end{eqnarray*}
therefore
$$\left\|u_n(\cdot,t)-u(\cdot,t)\right\|_{L^\infty(\mathcal D)}^p
\leq 
c^p_{n,p}\,T^{(p-2)/2} \int_0^t \left\|u_n(\cdot,s)-u(\cdot,s)\right\|^p_{L^\infty(\mathcal D)}\,ds \,, $$
hence, by Gr\"onwall's lemma, we obtain that $\left\|u_n(\cdot,t)-u(\cdot,t)\right\|_{L^\infty(\mathcal D)}=0,\; \forall t \in [0,\min\{T,\tau_M\}). $ It follows that $u_n(x,t,\omega)=u(x,t,\omega)$ on $\Omega_M^n$ a.s.\ if the problem \eqref{coint} is well posed.

Next we prove this section's main theorem which is the well-posedness of the cut-off problem \eqref{coint}.

\begin{theorem}\label{thmun}
Equation \eqref{coint} has a unique solution $u_n \in L^p(\Om_M;C([0,\min\{T,\tau_M\});\ch))$ for
$ p\geq 2. $ In particular, 
\begin{equation}\label{bound}
\normalfont \underset{t\in[0,\min\{T,\tau_M\})}{\text{sup}}E\left(\|u_n(\cdot,t)\|_{L^{\infty}(\cd)}^p\right) <+\infty.
\end{equation}
\end{theorem}
\begin{proof}We construct a Cauchy sequence, through a Picard iteration scheme, which converges, at a certain norm, to the solution $u_n$ of (\ref{coint}). For given $n\in\mN$, we set, $ \forall (x,t) \in \ocd\times [0,\min\{T,\tau_M\}),$
$$u_{n,0}(x,t):= \int_{\mathcal{D}}u_0(y)G(x,y,t)\,dy\ \ \text{for}\ t>0\quad \text{and}\quad u_{n,0}(x,0):=u_0(x) ,$$
and we define iteratively, for any integer $ k\geq 1,$
\begin{equation}\label{cointk}
\begin{split}
u_{n,k+1}(x,t) :=& \int_{\mathcal{D}}u_0(y)G(x,y,t)\,dy + \int_0^t\int_{\mathcal{D}}u_y(0,s)G_y(x,y,t-s)yT_{n}\bigl{(}y^{-1}u_{n,k}(y,s)\bigr{)}\,dy\,ds
\\[0.2cm]
&+\int_0^t\int_{\cd}G(x,y,t-s)\sigma(y)W(dy,ds).
\end{split}
\end{equation}

\noindent
Therefore, for $k\geq1$,
$$u_{n,k+1}(x,t)-u_{n,k}(x,t)=\int_0^t\int_{\mathcal{D}}u_{y}(0,s)G_y(x,y,t-s)y\left(T_n\bigl{(}y^{-1}u_{n,k}(y,s)\bigr{)}-T_n\bigl{(}y^{-1}u_{n,k-1}(y,s)\bigr{)}\right) dy\,ds ,$$
and applying the $\|\cdot\|_{\ch}$-norm at both sides,
then taking $p$-powers,
next taking the supremum on $ t\in[0,\min\{T,\tau_M\}), $ and then taking expectation, we obtain for any $k\in\mN$,
\begin{flalign}\label{bo1}
&E\biggl(\underset{t\in[0,\min\{T,\tau_M\})}{\text{sup}}\bigg[\|u_{n,k+1}(\cdot,t)-u_{n,k}(\cdot,t)\|_{\ch}^p\bigg]\biggr)=
\\
&E\biggl(\sup_{t\in[0,\min\{T,\tau_M\})} \bigg[\
\bigg\|\int_{0}^{t}\int_{\mathcal{D}}u_{y}(0,s)G_y(\cdot,y,t-s)y
\left(T_n\bigl(y^{-1}u_{n,k}(y,s)\bigr)-T_n\bigl(y^{-1}u_{n,k-1}(y,s)\bigr)\right) dy\,ds\bigg\|_{\ch}^p\
\bigg]\biggr).\nonumber
\end{flalign}

Since 
$T_n(\tilde{v})$ is uniformly Lipschitz on $\tilde{v}$, 
we have that
$\left|T_n(y^{-1}v)-T_n(y^{-1}w)\right|\leq c y^{-1}|v-w|$, for any $y>0$, and taking the supremum on $y\in\mathcal D, $ we get
$$\left\|T_n(y^{-1}v)-T_n(y^{-1}w)\right\|_{L^\infty(\mathcal D)}\leq
c\|v-w\|_{\mathcal{H}}.$$
In the right-hand side of \eqref{bo1}, we use the inequality 5 in \cite[Proposition 4.2]{BH} and the inequality $|u_y(0,s)|\leq M, $ then we apply H\"older's inequality and the above inequality, to obtain, for $ p>2,$

\begin{flalign*}
&E \biggl(\underset{t\in[0,\min\{T,\tau_M\})}{\text{sup}}\bigg[\left\|u_{n,k+1}(\cdot,t)-u_{n,k}(\cdot,t)\right\|_{\ch}^p \bigg]\biggr) \leq
\\[0.5em]
&
\leq M^p E\Bigg(\sup_{t\in[0,\min\{T,\tau_M\})} \Bigg[
\Bigg(\int_0^t \left(\sup_{x\in\mathcal D}\int_{\mathcal{D}} \big|G_y(x,y,t-s)\big|\,\frac{y}{x}\,dy\right)\;\times
\\[0.5em]
&
\phantom{\leq M^p E\biggl(\sup_{t\in[0,\min\{T,\tau_M\})} \bigg[}
\times \left\|T_n\bigl(y^{-1}u_{n,k}(y,s)\bigr)-T_n\bigl(y^{-1}u_{n,k-1}(y,s)\bigr)\right\|_{L^\infty(\mathcal D)}\,ds\Bigg)^p\Bigg]\Bigg)
\\[0.5em]
&
\leq c(T,M,p) \sup_{t\in[0,\min\{T,\tau_M\})}\Bigg[
\left(\int_0^t \left((t-s)^{-1/2}\right)^{\frac{p}{p-1}}ds\right)^{p-1}\,\times
\\[0.5em]
& \phantom{\leq c(T,M,p) \sup_{t\in[0,\min\{T,\tau_M\})}\Big(}\times E\left(\int_0^t
\left\|T_n\bigl(y^{-1}u_{n,k}(y,s)\bigr)-T_n\bigl(y^{-1}u_{n,k-1}(y,s)\bigr)\right\|^p_{L^\infty(\mathcal D)}\,ds\right)\Bigg]
\\[0.5em]
& 
\leq c(T,M,p)\int_0^{\min\{T,\tau_M\}}\hspace*{-0.1cm}E\biggl(\underset{t\in[0,s)}{\text{sup}}\bigg[\big\|u_{n,k}(\cdot,t)-u_{n,k-1}(\cdot,t)\big\|_{\ch}^p\,
\bigg]\biggr)\, ds.
\end{flalign*}
Applying recursively the above inequality for the integrand in the right-hand side, we get
\begin{align*}
\|&u_{n,k+1}(\cdot,t)-u_{n,k}(\cdot,t)\|^p_{L^p(\Om_M;C([0,\min\{T,\tau_M\});\ch))}
\\
&=E\biggl{(}\underset{t\in[0,\min\{T,\tau_M\})}{\text{sup}}\bigg{[}\|u_{n,k+1}(\cdot,t)-u_{n,k}(\cdot,t)\|_{\ch}^{p} \bigg{]}\biggr{)}
\\
&\leq c(T,M,p)^k\int_0^{\min\{T,\tau_M\}}\int_0^{s_k}\int_{0}^{s_{k-1}}\dots\int_0^{s_2}
E\biggl{(}\underset{t\in[0,s_1)}{\text{sup}}\bigg{[}\big{\|}u_{n,1}(\cdot,t)-u_{n,0}(\cdot,t)\big{\|}_{\ch}^p\ \bigg{]}\biggr{)}\ ds_1\dots ds_k
\\
&\leq \|u_{n,1}-u_{n,0}\|^p_{L^p(\Om_M;C([0,\min\{T,\tau_M\});\ch))}\
c(T,M,p)^k\ \dfrac{(\min\{T,\tau_M\})^{k}}{k!}\rightarrow0,\qquad \text{as}\,\;k\to+\infty.
\end{align*}
Therefore, $\{u_{n,k}\}_{k\in\mN}$ is a Cauchy sequence in $L^p\bigl{(}\Om_M;C([0,\min\{T,\tau_M\});\ch)\bigr{)}$ and so, by the completeness of the Banach space $\ch$ in this norm, it converges in $L^p\bigl{(}\Om_M;C([0,\min\{T,\tau_M\});\ch)\bigr{)}$ to some unique $u_n\in L^p\bigl{(}\Om_M;C([0,\min\{T,\tau_M\});\ch)\bigr{)}$ as $k\rightarrow+\infty$.

Since $u_{n,k}\rightarrow u_n$ in the norm $L^p\bigl{(}\Om_M;C([0,\min\{T,\tau_M\});\ch)\bigr{)}$ as
$k\rightarrow+\infty$ and since the mapping $T_n$ is Lipschitz (and therefore uniformly continuous), by a standard argument, where we take limits in the $L^{p}\bigl{(}\Om_M;C([0,\min\{T,\tau_M\});\ch)\bigr{)}$ norm in equation \eqref{cointk}, we conclude that $u_{n}$ satisfies equation
\eqref{coint}.

Next, we show the uniqueness of solution of \eqref{coint}. If we suppose that there exists another solution $ w_n $ of \eqref{coint}, then by subtracting their respective equations we get
$$u_n(x,t)-w_n(x,t)=\int_0^t\int_{\mathcal{D}}u_y(0,s) G_y(x,y,t-s) y \left(T_n\bigl{(}y^{-1}u_n(y,s)\bigr{)}-T_n\bigl{(}y^{-1}w_n(y,s)\bigr)\right) dy\,ds ,$$
and, by following a similar process as above, we obtain
$$E\Big(\sup_{t\in[0,\min\{T,\tau_M\})} \Big[\|u_n(\cdot,t)-w_n(\cdot,t)\|_{\ch}^p\Big]\Big)
\leq c(T,M,p) \int_0^{\min\{T,\tau_M\}}\hspace*{-0.1cm}E\Big(\sup_{t\in[0,s)}\Big[\big{\|}u_n(\cdot,t)-w_n(\cdot,t)\big{\|}_{\ch}^p\Big]\Big)\, ds.$$ Hence, by applying Gr\"onwall's Lemma to
the previous inequality, we get
$$E\Big(\sup_{t\in[0,\min\{T,\tau_M\})} \Big[\|u_n(\cdot,t)-w_n(\cdot,t)\|_{\ch}^p\Big]\Big)
\leq 0\ \ \ \text{and so }\ \
E\Big(\sup_{t\in[0,\min\{T,\tau_M\})} \Big[\|u_n(\cdot,t)-w_n(\cdot,t)\|_{\ch}^p
\Big]\Big) = 0.$$ This yields $u_n(x,t)=w_n(x,t)$
almost surely in $\Om_M$ and in $\Om_M^{n}$ for any
$(x,t)\in\ocd\times[0,\min\{T,\tau_M\})$, i.e.,
$$P\left(\big{\{}\om\in \Om_M\ (\text{or}\ \Om_M^n):\; u_n(x,t;\om)=w_n(x,t;\om)\big{\}}\right)=1,\quad \forall (x,t)\in\ocd\times[0,\min\{T,\tau_M\}),$$
and so by definition $u_n,\ w_n$ are equivalent in $\Om_M$
and in $\Om_M^n$.

By Theorem \ref{thmcon}, we know that $u$ has almost surely in
$\Omega_M$ continuous trajectories in $\ocd\times[0,\min\{T,\tau_M\})$ and the approximations $u_n,\ w_n$ satisfy equation \eqref{int} almost surely in $\Om_M^{n}$.
So, $u_n,\ w_n$ are also almost surely continuous in
$\Om_M^n$ and thus indistinguishable in $\Om_M^n$, i.e.,
$$P\left(\big{\{}\om\in \Om_M^n:\; u_n(x,t;\om)=w_n(x,t;\om),\ \forall (x,t)\in\ocd\times[0,\min\{T,\tau_M\})\big{\}}\right)=1,$$
which proves the uniqueness of solution of equation \eqref{int} with uniquely defined paths almost surely on $\Om_M^n$.

We finally prove \eqref{bound}. Since $u_{n,k}\rightarrow u_n$ in $L^p\bigl{(}\Om_M;C([0,\min\{T,\tau_M\});\ch)\bigr{)}$ as $k\rightarrow+\infty$, there exist positive constants $ c_1=c_1(n,p,M,T) $ and $ c_2=c_2(n,p,M,T) $ such that, for any $k\in\mathbb N,$
$$\|u_{n,k}-u_n\|^p_{L^p(\Om_M;C([0,\min\{T,\tau_M\});\ch))}\leq c_1 ,
\qquad \text{and} \qquad \|u_{n,k}\|^p_{L^p(\Om_M;C([0,\min\{T,\tau_M\});\ch))}\leq c_2.$$
Noticing that $ \|f\|_{L^{\infty}(\cd)} \leq \lambda \|f\|_{\ch} $ for any $f\in\ch, $ we get that, for any $t\in[0,\min\{T,\tau_M\}),$
\begin{eqnarray*}
E\big(\|u_{n,k}(\cdot,t)-u_n (\cdot,t)\|^p_{L^{\infty}(\cd)}\big)
&\leq& E\big(\underset{t\in[0,\min\{T,\tau_M\})}
{\text{sup}}\|u_{n,k}(\cdot,t)-u_n (\cdot,t)\|^p_{L^{\infty}(\cd)}\big)
\\
&\leq &\lambda^p \|u_{n,k}-u_n\|^p_{L^p(\Om_M;C([0,\min\{T,\tau_M\});\ch))}
\\[0.2cm]
&\leq &\lambda^p c_1,
\end{eqnarray*}
and 
\begin{eqnarray}
E\big(\|u_{n,k}(\cdot,t)\|^p_{L^{\infty}(\cd)}\big)
&\leq&
 E\big(\underset{t\in[0,\min\{T,\tau_M\})}
{\text{sup}}\|u_{n,k}(\cdot,t)\|^p_{L^{\infty}(\cd)}\big)
\nonumber
\\
&\leq& \lambda^p \|u_{n,k}\|^p_{L^p(\Om_M;C([0,\min\{T,\tau_M\});\ch))} 
\nonumber
\\
&\leq& \lambda^p c_2.\label{kwstas2}
\end{eqnarray}
Therefore, we obtain that
\begin{eqnarray*}
E\big(\|u_n(\cdot,t)\|^p_{L^{\infty}(\cd)}\big) 
&\leq& 2^{p-1} E\big(\|u_n(\cdot,t)-u_{n,k}(\cdot,t)\|^p_{L^{\infty}(\cd)}\big)
+
2^{p-1} E\big(\|u_{n,k}(\cdot,t)\|^p_{L^{\infty}(\cd)}\big)
\\[0.2cm]
&\leq &  2^{p-1} \lambda^p (c_1+c_2) ,
\end{eqnarray*}
for any $t\in[0,\min\{T,\tau_M\}), $ and hence, for $p>2,$
$$\sup_{t\in[0,\min\{T,\tau_M\})}E\big(\|u_n(\cdot,t)\|^p_{L^{\infty}(\cd)}\big)<c(n,p,M,T,\lambda) <+\infty.$$
For $ p=2, $ \eqref{bound} follows by H\"older inequality.
\end{proof}

\section{The Malliavin derivative of the solution $u$ for the unreflected problem}
We will prove that the Malliavin derivative of solution $u_n $ of equation \eqref{coint} for $\eta:=0$, is well defined as the solution of an s.p.d.e. Moreover, we shall prove its regularity in
$D^{1,2}$ and in $L^{1,2}$ by the next proposition.

In what follows, the notation $D_{y,s}f(x,t)$ for a general
function $f$ is used to denote $D_{y,s}(f(x,t))$ which is a
function of the variables $y,s,x,t$. Additionally due to the
independence of the variables $y,s$ from $x,t$ we may commute the
limit as $x\rightarrow 0$ with $D_{y,s}$ if the terms are well
defined, i.e., $D_{y,s}(\displaystyle{\lim_{x\rightarrow
0}}f(x,t))=\displaystyle{\lim_{x\rightarrow 0}}D_{y,s}(f(x,t))$.
So, as $x$ is not depending on the realization obviously, and as
the Malliavin derivative is linear, we have, when the terms are
well defined, that
$D_{y,s}(u_x(0^+,t))=D_{y,s}(\displaystyle{\lim_{x\rightarrow
0^+}}x^{-1}u(x,t))=\displaystyle{\lim_{x\rightarrow
0^+}}x^{-1}D_{y,s}(u(x,t))=:\nabla_x(D_{y,s}u)(0^+,t)$.

\begin{proposition}\label{malun}
Let $u_{n}$ be the unique solution of equation \eqref{coint}. Let
$M_{\rm d}>0$ a fixed deterministic value. We define the stopping
time $\tau_{M_{\rm d }}:=\inf\{T>
0:\;\displaystyle{\sup_{r\in[0,T)}}|D_{y,s}(u_x(0^+,r))|\geq
M_{\rm d}\}$ (note that at $r=0$ $u_x(0^+,0)=(u_0)_x(0^+)$ is
deterministic, and so its Malliavin derivative is zero and so
less than $M_{\rm d}$). Also let $T>0$ be a deterministic time
value.

It holds that:
\begin{itemize}
\item[(i)]$u_n(x,t)\in D^{1,2}\bigl{(}\cd\times[0,\min\{T,\tau_M,\tau_{M_{\rm d}}\})\bigr{)}$.

\vspace{1em}

\item[(ii)]The Malliavin derivative of $u_{n}$ satisfies for all
$(x,t)\in\cd\times[0,\min\{T,\tau_M,\tau_{M_{\rm d}}\})$ and for
any $s\leq t$, uniquely the stochastic equation 
\begin{equation}\label{41}
\begin{split}
D_{y,s}u_n(x,t)=\;&G(x,y,t-s)\sigma(y)+\int_s^t\int_{\cd}G_{\Tilde{y}}(x,\Tilde{y},t-\Tilde{s})u_{\Tilde{y}}(0,\Tilde{s})
 \ \cg_{n}(\Tilde{y},\Tilde{s})D_{y,s}u_n(\Tilde{y},\Tilde{s})\,d\Tilde{y}\,d\Tilde{s}
\\
&+\int_s^t\int_{\cd}G_{\Tilde{y}}(x,\Tilde{y},t-\Tilde{s})
D_{y,s}(u_{\Tilde{y}}(0,\Tilde{s}))\Tilde{y}T_n(\Tilde{y}^{-1}u_n(\Tilde{y},\Tilde{s}))\,d\Tilde{y}\,d\Tilde{s},
\end{split}
\end{equation}
and also for any $s>t$ it holds that $D_{y,s}u_n(x,t)=0$. In
the above, $\cg_{n}$ is a stochastic process which satisfies, for
$c_L>0$ a deterministic constant, that
\begin{equation}\label{42}
\begin{split}
&D_{y,s}\left(T_n\bigl{(}x^{-1}u_n(x,t)\bigr)\right)=x^{-1}
\cg_n(x,t)D_{y,s}\bigl(u_n(x,t)\bigr),\;\;\mbox{where}\\
&|\cg_{n}(x,t)|\leq c_L,\ \forall (x,t)\in\cd\times[0,\infty) \ \text{a.s.}
\end{split}
\end{equation}
\item[(iii)] $u_n\in L^{1,2}\bigl(\cd\times[0,\min\{T,\tau_M,\tau_{M_{\rm
d}}\})\bigr)$.
\end{itemize}
\end{proposition}
\begin{proof} (i) We will prove that the Cauchy sequence
$\{u_{n,k}\}_{k\in\mN}$ (which is defined in Theorem \ref{thmun})
belongs to $ D^{1,2}\bigl(\cd\times[0,\min\{T,\tau_M,\tau_{M_{\rm d}}\})\bigr),\ \forall (x,t)\in \cd\times[0,\min\{T,\tau_M,\tau_{M_{\rm d}}\})$, by using induction and the Picard iteration scheme \eqref{cointk}.

For $k=0$, the function $ u_{n,0} $ is deterministic and bounded $($since $u_0\in\ch)$ with Malliavin derivative $Du_{n,0}=0$ and so $u_{n,0}\in D^{1,2}$.

Let $k\in\mN$ be fixed and $2<p<3. $ We suppose that, $\forall
i\in\mN\cap[0,k]$, it holds that
\begin{equation}
\begin{split}
&u_{n,i}(x,t)\in D^{1,p},\ \forall (x,t)\in\ocd\times[0,\min\{T,\tau_M,\tau_{M_{\rm d}}\}),\;\;\mbox{and}
\\
&\sup_{t\in[0,\min\{T,\tau_M,\tau_{M_{\rm d}}\})} \sup_{i\leq k}
\bigg[E\biggl{(}\int_0^t\int_{\cd}\|D_{y,s}u_{n,i}(\cdot,t)\|_{L^p(\cd)}^p\,dy\,ds\biggr{)}\bigg]<+\infty.
\end{split}
\end{equation}
We shall prove that the above also holds for $u_{n,k+1}$, i.e.,
$$    u_{n,k+1}(x,t)\in D^{1,p},\ \forall
(x,t)\in\ocd\times[0,\min\{T,\tau_M,\tau_{M_{\rm d}}\}),$$ and
$$\sup_{t\in[0,\min\{T,\tau_M,\tau_{M_{\rm d}}\})} \sup_{i\leq
k+1}
\bigg{[}E\biggl{(}\int_0^t\int_{\cd}\|D_{y,s} u_{n,i}(\cdot,t)\|_{L^p(\cd)}^p\,dy\,ds\biggr{)}\bigg{]}<+\infty,$$
where the bounds are independent of $k$.

Since the Malliavin derivative $D:=D_{\cdot,\cdot}$ is a linear
operator, and since $D(uv)=D(u)v+uD(v)$, and $D(u)=0$ when $u$ is
not depending on the realization, applying it at both sides of
\eqref{cointk}, we get
\begin{equation*}
\begin{split}
D_{y,s}u_{n,k+1}(x,t) =& D_{y,s}\biggl{(}\int_{\mathcal{D}}u_{0}(\Tilde{y})G(x,\Tilde{y},t)\,d\Tilde{y}\biggr{)}
+D_{y,s}\biggl{(}\int_{0}^{t}\int_{\mathcal{D}}u_{\Tilde{y}}(0,\Tilde{s})G_{\Tilde{y}}(x,\Tilde{y},t-\Tilde{s})\Tilde{y}T_{n}\bigl{(}\Tilde{y}^{-1}u_{n,k}
(\Tilde{y},\Tilde{s})\bigr{)}\,d\Tilde{y}\,d\Tilde{s}\biggr{)}\\
&+D_{y,s}\biggl{(}\int_{0}^{t}\int_{\cd}G(x,\Tilde{y},t-\Tilde{s})\sigma(\Tilde{y})W(d\Tilde{y},d\Tilde{s})\biggr{)}
\\
=&0+\int_{s}^{t}\int_{\cd}G_{\Tilde{y}}(x,\Tilde{y},t-\Tilde{s})\ D_{y,s}
    \bigl{(}u_{\Tilde{y}}(0,\Tilde{s})\Tilde{y}T_{n}\bigl{(}\Tilde{y}^{-1}u_{n,k}(\Tilde{y},\Tilde{s})\bigr{)}\bigr{)}\,d\Tilde{y}\,d\Tilde{s}+G(x,y,t-s)\sigma(y)\\
   = & \int_{s}^{t}\int_{\cd}G_{\Tilde{y}}(x,\Tilde{y},t-\Tilde{s})\
u_{\Tilde{y}}(0,\Tilde{s})\Tilde{y}D_{y,s}
\bigl{(}T_{n}\bigl{(}\Tilde{y}^{-1}u_{n,k}(\Tilde{y},\Tilde{s})\bigr{)}\bigr{)}\,d\Tilde{y}\,d\Tilde{s}
\\
&+\int_s^t\int_{\cd}G_{\Tilde{y}}(x,\Tilde{y},t-\Tilde{s})
D_{y,s}(u_{\Tilde{y}}(0,\Tilde{s}))\Tilde{y}T_n(\Tilde{y}^{-1}u_{n,k}(\Tilde{y},\Tilde{s}))\,d\Tilde{y}\,d\Tilde{s}+G(x,y,t-s)\sigma(y),
\end{split}
\end{equation*}
where we used \cite[Proposition 1.3.8]{nual}, for the stochastic integral term and the fact that the Malliavin derivative is zero when applied to the deterministic terms $ u_0,\, G,\, G_{\Tilde{y}}$, $\Tilde{y}$, and also zero for any
$\Tilde{s}<s$.

By the induction hypothesis $u_{n,k}(x,t)\in D^{1,2}$ and recall that since $H_n$ is $C^1$ with $ |H_n'|\leq 2, $ then for arbitrary $\tilde{v}$, $T_n(\tilde{v})$ is uniformly Lipschitz on $\tilde{v}$, with (deterministic) Lipschitz constant $c_L$, and so it holds that $|T_n(x^{-1}v)-T_n(x^{-1} w)|\leq c_L|x^{-1}v-x^{-1}w|, $ hence \cite[Proposition 1.2.4]{nual} (analogous to chain rule) ensures that
$$T_n\bigl{(}x^{-1}u_{n,k}(x,t)\bigr{)}\in D^{1,2}\bigl{(}\cd\times[0,\min\{T,\tau_M,\tau_{M_{\rm d}}\})\bigr{)} ,$$
and there exists a stochastic process $\mathcal{G}_{n,k}(x,t)$ such that
$$ D_{y,s}\bigl(T_n\bigl(x^{-1} u_{n,k}(x,t)\bigr)\bigr)=\cg_{n,k}(x,t)D_{y,s}(x^{-1}u_{n,k}(x,t)) = \cg_{n,k}(x,t) x^{-1} D_{y,s}(u_{n,k}(x,t)),$$
with $$|\cg_{n,k}(x,t)|\leq c_L,\; \forall x\in\ocd,\, t\in[0,\infty)\ \text{a.s.},
$$
where we used that $x^{-1}$ is deterministic and has thus zero Malliavin derivative, so $D(x^{-1}v)=x^{-1}D(v)$.

The above yields, for all $ s\leq t$,
\begin{equation}\label{dk1}
\begin{split}
D_{y,s}u_{n,k+1}(x,t)=&\int_{s}^{t}\int_{\cd}G_{\Tilde{y}}(x,\Tilde{y},t-\Tilde{s})u_{\Tilde{y}}(0,\Tilde{s})
\ \cg_{n,k}(\Tilde{y},\Tilde{s})D_{y,s}u_{n,k}(\Tilde{y},\Tilde{s})\,d\Tilde{y}\,d\Tilde{s}
\\
&+\int_{s}^{t}\int_{\cd}G_{\Tilde{y}}(x,\Tilde{y},t-\Tilde{s})
D_{y,s}(u_{\Tilde{y}}(0,\Tilde{s}))\Tilde{y}T_n(\Tilde{y}^{-1}u_{n,k}(\Tilde{y},\Tilde{s}))\,
d\Tilde{y} \, d\Tilde{s}+G(x,y,t-s)\sigma(y),
\end{split}
\end{equation}
and for any $ s>t, $
$$D_{y,s}u_{n,k+1}(x,t)=0.$$
We have $ T_n(v)=H_n(v)v $ for general
$ v, $ and $ H_n\leq 1$, hence for $v:=\Tilde{y}^{-1}u_{n,k}(\Tilde{y},\Tilde{s})$, it holds
\begin{equation*}
\big|\Tilde{y} \, T_n(\Tilde{y}^{-1}u_{n,k}(\Tilde{y},\Tilde{s}))\big| = \big|\Tilde{y}\, H_n\!\left(\Tilde{y}^{-1}\,u_{n,k}(\Tilde{y},\Tilde{s})\right)\,\Tilde{y}^{-1} \, u_{n,k}(\Tilde{y},\Tilde{s})\big|
= H_n\!\left(\Tilde{y}^{-1}u_{n,k}(\Tilde{y},\Tilde{s})\right)\, |u_{n,k}(\Tilde{y},\Tilde{s})|\\
\leq  |u_{n,k}(\Tilde{y},\Tilde{s})| .
\end{equation*}
Also, we are in the sample space $\Omega_M$ where $|u_y(0,t)|<M$ for $t< \tau_M. $ 
Therefore, taking absolute value in \eqref{dk1} and raising to $p $ power, we obtain
\begin{equation*}
\begin{split}
\big|D_{y,s}u_{n,k+1}(x,t)\big|^p \leq & c_p \biggl(\int_s^t\int_{\cd}\big|G_{\Tilde{y}}(x,\Tilde{y},t-\Tilde{s})\big|\,
\big|u_{\Tilde{y}}(0,\Tilde{s})\big| \,
\big|\cg_{n,k}(\Tilde{y},\Tilde{s})\big|\,
\big|D_{y,s}\bigl(u_{n,k}(\Tilde{y},\Tilde{s})\bigr)\big|\,
d\Tilde{y}\,d\Tilde{s}\biggr)^p
\\
&+ c_p \biggl{(}\int_s^t\int_{\cd}\big{|}G_{\Tilde{y}}(x,\Tilde{y},t-\Tilde{s})\big{|}
\big{|}D_{y,s}(u_{\Tilde{y}}(0,\Tilde{s}))\big{|}|u_{n,k}(\Tilde{y},\Tilde{s})|\,
d\Tilde{y}\,d\Tilde{s}\biggr)^p +\, c_p\big|G(x,y,t-s)\sigma(y)\big|^p\\
&\leq c_p\, c_L^p\, M^p \,\biggl(\int_s^t\int_{\cd}\big|G_{\Tilde{y}}(x,\Tilde{y},t-\Tilde{s})\big|\
\big|D_{y,s}\bigl(u_{n,k}(\Tilde{y},\Tilde{s})\bigr)\big|\,
d\Tilde{y}\,d\Tilde{s}\biggr)^p
\\
&+c_p\, M_{\rm d}^p\,\biggl(\int_s^t\int_{\cd}\big|G_{\Tilde{y}}(x,\Tilde{y},t-\Tilde{s})\big|\,
|u_{n,k}(\Tilde{y},\Tilde{s})|\,
d\Tilde{y}\,d\Tilde{s}\biggr)^p\,+\,c_p\,\big|G(x,y,t-s)\,\sigma(y)\big|^p.
\end{split}
\end{equation*}
We integrate the above for $y\in\cd,\ s\in[0,t]$ and then take expectation to derive
\begin{equation}\label{eqn1}
\begin{split}
E\biggl(\int_0^t\int_{\cd}&\left|D_{y,s}u_{n,k+1}(x,t)\right|^p\, dy\,ds\biggr)
\leq c_p \int_0^t\int_{\cd}\big|G(x,y,t-s)\,\sigma(y)\big|^p\, dy\,ds\\
&+ c_p \, c_L^p\,M^p\, E\biggl(\int_0^t\int_{\cd}\bigg(\int_s^t\int_{\cd}\big|G_{\Tilde{y}}(x,\Tilde{y},t - \Tilde{s})\big|\, \big|D_{y,s}\bigl(u_{n,k}(\Tilde{y},\Tilde{s})\bigr)\big|\, d\Tilde{y}\,d\Tilde{s}\bigg)^p dy\,ds\biggr)\\
&+ c_p \,M_{\rm d}^p \,E\biggl(\int_0^t\int_{\cd}\bigg(\int_s^t\int_{\cd}\big|G_{\Tilde{y}}(x,\Tilde{y},t-\Tilde{s})\big|\, |u_{n,k}(\Tilde{y},\Tilde{s})|\, d\Tilde{y}\,d\Tilde{s}\bigg)^p dy\,ds\biggr).
\end{split}
\end{equation}

For the first term in the right-hand side of \eqref{eqn1}, we use \eqref{ggg} and \eqref{kostasgaus} to get, for $p<3,$
\begin{flalign}
\int_0^t\int_{\cd}\big|G(x,y,t-s)\,\sigma(y)\big|^p\, dy\,ds\
&\leq \|\sigma\|^p_{L^{\infty}(\cd)} \int_0^t\int_{\cd}\big|G(x,y,t-s)\big|^p\, dy\,ds
\nonumber
\\[0.2cm]
&\leq c\, \|\sigma\|^p_{L^{\infty}(\cd)} \int_0^t(t-s)^{\frac{1-p}{2}}\ ds
\nonumber
\\[0.2cm]
&= \frac{2 c}{3-p} \, \|\sigma\|^p_{L^{\infty}(\cd)} \, t^{(3-p)/2} < + \infty ,
\label{Gest1}
\end{flalign}
where $c$ is a positive constant.

For the second term in the right-hand side of \eqref{eqn1}, we first use \eqref{gen}, then we apply H\"older inequality with exponents $p $ and $ p'=\tfrac{p}{p-1}$ in the integral w.r.t.\ $ \tilde{y} $ and we use \eqref{kostasgaus}, then we choose a $ \beta>(3-p)/(2p) $ and apply H\"older inequality in the integral w.r.t.\ $ \tilde{s}, $ to get
\allowdisplaybreaks
\begin{align*}
\bigg(\int_s^t\int_{\cd}\big{|}G_{\Tilde{y}}(x,\Tilde{y},t-\Tilde{s})\big{|}&\ \big{|}D_{y,s}\bigl{(}u_{n,k}(\Tilde{y},
\Tilde{s})\bigr{)}\big{|}\ d\Tilde{y}\,d\Tilde{s}\bigg)^p \leq
\\[0.2cm]
&\leq\biggl(\int_s^t\int_{\cd}\dfrac{c_1}{|t-\Tilde{s}|}\; \text{exp}\biggl(-\frac{c_2|x-\Tilde{y}|^2}{|t-\tilde{s}|}\biggr)
\ \big|D_{y,s}\bigl(u_{n,k}(\Tilde{y},\Tilde{s})\bigr)\big|\,d\Tilde{y}\,d\Tilde{s}\biggr)^p
\\[0.2cm]
&\leq\biggl(\int_s^t\dfrac{c_1}{|t-\Tilde{s}|} \; \big\|D_{y,s}\bigl(u_{n,k}(\cdot,\Tilde{s})\bigr)\big\|_{L^p(\cd)}
\left(\int_{\cd}\text{exp}\biggl{(}-\frac{c_2 p' |x-\Tilde{y}|^2}{|t-\tilde{s}|}\biggr)\, d\Tilde{y}\right)^{1/p'} d\Tilde{s}\biggr)^p
\\[0.2cm]
&\leq c_3 \, \biggl(\int_s^t\dfrac{1}{|t-\Tilde{s}|} \, \big\|D_{y,s}\bigl(u_{n,k}(\cdot,\Tilde{s})\bigr) \big\|_{L^p(\cd)}\, |t-\Tilde{s}|^{\frac{1}{2p'}} \, d\Tilde{s}\biggr)^p
\\[0.2cm]
&=c_3 \biggl{(}\int_s^t |t-\Tilde{s}|^{\frac{1}{2p'}-1} \, \big\|D_{y,s}\bigl(u_{n,k}(\cdot,\Tilde{s})\bigr) \big\|_{L^p(\cd)} \, d\Tilde{s}\biggr)^p
\\[0.2cm]
&\leq
c_3 \, \left(\int_s^t |t-\Tilde{s}|^{\frac{1}{2}+(\beta-1)p'}\, d\Tilde{s}\right)^{p-1}
\int_s^t |t-\Tilde{s}|^{-\beta p} \, \big{\|}D_{y,s}\bigl{(}u_{n,k}(\cdot,\Tilde{s})\bigr{)} \big\|^p_{L^p(\cd)}d\Tilde{s}
\\[0.2cm]
&= c_3\,c(p,\beta,t,s) \int_s^t  |t-\Tilde{s}|^{-\beta p} \,  \big{\|}D_{y,s}\bigl{(}u_{n,k}(\cdot,\Tilde{s})\bigr{)} \big\|^p_{L^p(\cd)}d\Tilde{s}
\\[0.2cm]
&\leq c_4 \int_0^t  |t-\Tilde{s}|^{-\beta p} \, \big{\|}D_{y,s}\bigl{(}u_{n,k}(\cdot,\Tilde{s})\bigr{)} \big\|^p_{L^p(\cd)}d\Tilde{s} ,
\end{align*}
with $ c(p,\beta, t,s)=\left(\tfrac{(t-s)^{3/2+(\beta-1)p'}}{3/2+(\beta-1)p'}\right)^{p-1} \leq \left(\tfrac{T^{3/2+(\beta-1)p'}}{3/2+(\beta-1)p'}\right)^{p-1}  =:c_4/c_3. $ 

Thus, it follows
\begin{align}
E\biggl(\int_0^t\int_{\cd}\bigg(\int_s^t\int_{\cd}\big|G_{\Tilde{y}}(x,\Tilde{y},t-\Tilde{s})&\big|\ \big|D_{y,s}\bigl(u_{n,k}(\Tilde{y},\Tilde{s})\bigr)\big|\, d\Tilde{y}\,d\Tilde{s}\bigg)^p\, dy\,ds\biggr)\leq
\nonumber\\[0.2cm]
&\leq  c_4 E\biggl(\int_0^t\int_{\cd}\int_0^t  |t-\Tilde{s}|^{-\beta p} \, \big\|D_{y,s}\bigl(u_{n,k}(\cdot,\Tilde{s})\bigr) \big\|^p_{L^p(\cd)}\ d\Tilde{s}\, dy\,ds\biggr)
\nonumber\\[0.2cm]
&=
c_4 \int_0^t E\biggl(\int_0^t  |t-\Tilde{s}|^{-\beta p} \, \int_{\cd}\big\|D_{y,s}\bigl(u_{n,k}(\cdot,\Tilde{s})\bigr) \big\|^p_{L^p(\cd)}\, dy\,ds\biggr)\,d\Tilde{s}\nonumber\\[0.2cm]
&=
c_4
\, \int_0^t E\biggl(\int_0^{\Tilde{s}}\int_{\cd}  |t-\Tilde{s}|^{-\beta p} \, \big{\|}D_{y,s}\bigl{(}u_{n,k}(\cdot,\Tilde{s})\bigr{)} \big{\|}^p_{L^p(\cd)}\, dy\,ds\biggr)\,d\Tilde{s} ,
\label{Gest2}
\end{align}
where we used the fact that the integral for $s$ is taken finally in $[0,\Tilde{s}]$, since for $s>\Tilde{s}$ the Malliavin derivative satisfies
$D_{y,s}\bigl{(}u_{n,k}(x,\Tilde{s})\bigr{)}=0$, for any
$x\in\cd$.

Regarding the last term at the right-hand side of \eqref{eqn1}, using \eqref{gy} and H\"older inequality, we obtain
\begin{align*}
E\biggl(\int_0^t\int_{\cd}\bigg|\int_s^t\int_{\cd}\big|G_{\Tilde{y}}(x,\Tilde{y},t-\Tilde{s})\big|\,
|u_{n,k}(\Tilde{y},\Tilde{s})|\, d\Tilde{y}\,d\Tilde{s}\bigg|^p dy\,ds\biggr) \leq
c\int_0^t\int_{\cd} E\biggl(\sup_{\tilde{s}\in[s,t]}\big\|u_{n,k}(\cdot,\Tilde{s})\big\|^p_{L^{\infty}(\cd)}\biggr)dy\,
ds,
\end{align*}
and then, employing the estimate \eqref{kwstas2}, we get
\begin{equation}\label{Gest3}
E\biggl(\int_0^t\int_{\cd}\bigg|\int_s^t\int_{\cd}\big|G_{\Tilde{y}}(x,\Tilde{y},t-\Tilde{s})\big|\,
|u_{n,k}(\Tilde{y},\Tilde{s})|\, d\Tilde{y}\,d\Tilde{s}\bigg|^p dy\,ds\biggr)
\leq c_5(n,p,\lambda,M,T)<+\infty.
\end{equation}

Next, integrating  \eqref{eqn1} with respect to $ x\in \mathcal D, $ and using the estimates \eqref{Gest1}, \eqref{Gest2} and \eqref{Gest3}, we obtain
\begin{multline}
E\biggl(\int_0^t\int_{\cd} \|D_{y,s}u_{n,k+1}(\cdot,t)\|^p_{L^p(\cd)}\,dy\,ds\biggr)\leq\\ c_6+ c_7\int_0^t E\biggl
(\int_0^{\Tilde{s}}\int_{\cd}  |t-\Tilde{s}|^{-\beta p} \, \big{\|}D_{y,s}\bigl(u_{n,k}(\cdot,\Tilde{s})\bigr) \big{\|}^p_{L^p(\cd)}\, dy\,ds\biggr)\, d\Tilde{s}.\label{Gest4}
\end{multline}
Taking supremum for $i\leq k$, we get
\begin{align*}
\underset{i\leq k}{\text{sup}}\bigg{[}E\biggl{(}\int_{0}^{t}\int_{\cd}\|D_{y,s}u_{n,i+1}(\cdot,t)&\|^p_{L^p(\cd)}\,dy\,ds\biggr{)}\bigg{]} \leq\\
&\leq c_6+ c_7\int_0^t\underset{i\leq k}{\text{sup}}\bigg{[}E\biggl{(}\int_{0}^{\Tilde{s}}\int_{\cd}  |t-\Tilde{s}|^{-\beta p} \, \big{\|}D_{y,s}\bigl{(}u_{n,i}(\cdot,\Tilde{s})\bigr{)}
\big{\|}^p_{L^p(\cd)}\, dy\,ds\biggr{)}\bigg{]} d\Tilde{s}\\[0.2cm]
&\leq c'+c\int_0^t\underset{i\leq k}{\text{sup}}\bigg{[}E\biggl{(}\int_{0}^{\Tilde{s}}\int_{\cd}  |t-\Tilde{s}|^{-\beta p} \, \big{\|}D_{y,s}
\bigl{(}u_{n,i+1}(\cdot,\Tilde{s})\bigr{)} \big{\|}^p_{L^p(\cd)}\,dy\,ds\biggr{)}\bigg{]}
d\Tilde{s},
\end{align*}
or equivalently,

\begin{equation}\label{dd1}
\begin{split}
\underset{i\leq k+1}{\text{sup}}\bigg{[}E\biggl{(}\int_0^t&\int_{\cd}\|D_{y,s}u_{n,i}(\cdot,t)\|^p_{L^p(\cd)}\,dy\,ds\biggr{)}\bigg{]}\leq\\
&\leq c' + c\int_{0}^{t}\underset{i\leq k+1}{\text{sup}}\bigg{[}E\biggl{(}\int_0^{\Tilde{s}}\int_{\cd}  |t-\Tilde{s}|^{-\beta p} \, \big{\|}D_{y,s}\bigl{(}u_{n,i}(\cdot,\Tilde{s})\bigr{)} \big{\|}^p_{L^p(\cd)}\, dy\,ds\biggr{)}\bigg{]} d\Tilde{s}.
\end{split}
\end{equation}

\noindent
We now define
\begin{align*}
A_{n,k+1}(t):= \underset{i\leq k+1}{\text{sup}}\bigg{[}E\biggl{(}\int_0^t\int_{\cd}\|D_{y,s}u_{n,i}(\cdot,t)\|^p_{L^p(\cd)}\,dy\,ds\biggr{)}\bigg{]},
\end{align*}
and \eqref{dd1} becomes
\begin{equation}\label{gron}
A_{n,k+1}(t) \leq c'+ c\int_0^t  |t-\Tilde{s}|^{-\beta p} \, A_{n,k+1}(\Tilde{s})\, d\Tilde{s}.
\end{equation}
By \eqref{gron} we get 
\begin{eqnarray}\label{gron2}
\int_0^t  |t-\Tilde{s}|^{-\beta p} \, A_{n,k+1}(\Tilde{s})\, d\Tilde{s}
&\leq&
c' \int_0^t  |t-\Tilde{s}|^{-\beta p} \, d\Tilde{s}
+
c\int_0^t \int_0^{\tilde{s}} |t-\Tilde{s}|^{-\beta p}\, |\tilde{s}-\tau|^{-\beta p} \, A_{n,k+1}(\tau)\, d\tau \, d\tilde{s}
\nonumber
\\[0.5em]
&=&
C_1
+
c \int_0^t \left( \int_{\tau}^{t} |t-\Tilde{s}|^{-\beta p}\, |\tilde{s}-\tau|^{-\beta p} \, d\tilde{s} \right) \, A_{n,k+1}(\tau) \, d\tau
\nonumber
\\[0.5em]
&\leq& C_1 + C_2 \int_0^t   A_{n,k+1}(\tau) \, d\tau , \label{gron3}
\end{eqnarray}
where we used that
$$ c'\int_0^t  |t-\Tilde{s}|^{-\beta p} \, d\Tilde{s} = \frac{c'\, t^{1- \beta p}}{1-\beta p} \leq \frac{c'\, T^{1- \beta p}}{1-\beta p} =: C_1  <+\infty , $$
and
\begin{eqnarray}
c\int_\tau^t|t-\Tilde{s}|^{-\beta p} \, |\tilde{s}-\tau|^{-\beta p} \,  \,d\tilde{s}
& \leq &
c\left(\int_\tau^t|t-\Tilde{s}|^{-2 \beta p} \,d\tilde{s}\right)^{1/2} 
\left( \int_\tau^t|\tilde{s}-\tau|^{-2 \beta p} \, \,d\tilde{s}\right)^{1/2}
\nonumber\\[0.5em]
&=&
c\left(\frac{1}{1-2 \beta p}\, |t-\tau|^{1-2 \beta p}\right)^{1/2} \, \left(\frac{1}{1-2 \beta p}\, |t-\tau|^{1-2 \beta p}\right)^{1/2}
\nonumber\\[0.5em]
&=&
\frac{c}{1-2 \beta p}\, |t-\tau|^{1-2 \beta p}
\nonumber\\[0.5em]
&\leq &
\frac{c\, T^{1-2 \beta p}}{1-2 \beta p} =: C_2  <+\infty ,\label{hineq}
\end{eqnarray}
for $ \beta<1/(2p). $ Using \eqref{gron3} in the last term of \eqref{gron}, we obtain
$$ A_{n,k+1}(t) \leq C_3 + C_4 \int_0^t A_{n,k+1}(\tau)\, d\tau , $$
so, by Gr\"onwall's lemma, we get
\begin{equation*}
\underset{i\leq k+1}{\text{sup}}\bigg{[}E\biggl{(}\int_{0}^{t}\int_{\cd}\|D_{y,s}u_{n,i}(\cdot,t)\|^p_{L^p(\cd)}\, dy\,ds\biggr{)}\bigg{]}=A_{n,k+1}(t)\leq
c(n,p,M,T),
\end{equation*}
therefore,
\begin{equation}\label{ind1}
\sup_{t\in[0,\min\{T,\tau_M,\tau_{M_{\rm d}}\})} \sup_{i\leq
k+1} \bigg{[}E\biggl{(}\int_0^t\int_{\cd}\|D_{y,s}u_{n,i}(\cdot,t)\|^p_{L^p(\cd)}\,
dy\,ds\biggr{)}\bigg{]}\leq c(n,p,M,T),
\end{equation}
which gives
$$\underset{t\in[0,\min\{T,\tau_M,\tau_{M_{\rm d}}\})}{\text{sup}}\underset{i\leq k+1}{\text{sup}}\bigg{[}E\biggl{(}\int_0^{\min\{T,\tau_M,\tau_{M_{\rm d}}\}}\hspace*{-0.2cm}
\int_{\cd}\|D_{y,s}u_{n,i}(\cdot,t)\|^p_{L^p(\cd)}\,dy\,ds\biggr{)}\bigg{]}\leq
 c(n,p,M,T).$$
In the previous, we used that $D_{y,s}u_{n,i}(x,t)=0,\ \forall
s>t$ and so the integration is for $ s\in[0,\min\{T,\tau_M,\tau_{M_{\rm d}}\})$, while we note that the bound is independent of $k$.

By the estimate \eqref{eqn1} combined with \eqref{Gest1}, \eqref{Gest2}, \eqref{Gest3} and \eqref{ind1}, we obtain

\begin{eqnarray}\label{kwstas3}
E\biggl(\int_0^{\min\{T,\tau_M,\tau_{M_{\rm d}}\}}\hspace*{-0.2cm}
\int_{\cd}|D_{y,s}u_{n,k+1}(x,t)|^p \,dy\,ds\biggr)
\nonumber
&=&
E\biggl(\int_0^t \int_{\cd}|D_{y,s}u_{n,k+1}(x,t)|^p dy\,ds\biggr)
\\
&\leq & C(n,p,M,T)<+\infty .
\end{eqnarray}

Then, we have that
\allowdisplaybreaks
\begin{align*}
&\|u_{n,k+1}(x,t)\|^{2}_{D^{1,2}(\cd\times[0,\min\{T,\tau_M,\tau_{M_{\rm d}}\}))}:=
E\bigl{(}|u_{n,k+1}(x,t)|^{2}\bigr{)}+E\biggl{(}\int_{0}^{\min\{T,\tau_M,\tau_{M_{\rm d}}\}}\hspace*{-0.2cm}\int_{\cd}|D_{y,s}u_{n,k+1}(x,t)|^2 dy ds\biggr{)}
\\[0.2cm]
&\leq  E\bigl(|u_{n,k+1}(x,t)|^p\bigr)^{2/p}+ c E\biggl(\int_0^{\min\{T,\tau_M,\tau_{M_{\rm d}}\}}\hspace*{-0.2cm}
\int_{\cd}|D_{y,s}u_{n,k+1}(x,t)|^p dy\,ds\biggr)^{2/p}\,,
\end{align*}
which is bounded uniformly for any $k\in\mN$, due to the estimates \eqref{kwstas2} and \eqref{kwstas3}. Therefore, $ u_{n,k+1}(x,t)\in
D^{1,2}(\cd\times[0,\min\{T,\tau_M,\tau_{M_{\rm
d}}\})) $ and the induction is completed.

Since $u_{n,k}\rightarrow u_{n}$ in
$L^p\bigl{(}\Om_M;C([0,\min\{T,\tau_M,\tau_{M_{\rm d}}\});\ch)\bigr{)}$ for $p>2$, we also have that
$u_{n,k}\rightarrow u_{n}$ in $L^2\bigl{(}\Om_M;C([0,\min\{T,\tau_M,\tau_{M_{\rm
d}}\});\ch)\bigr{)}, $ by H\"older's inequality on the expectation as $p>2$.

Furthermore, we've proven that $u_{n,k}(x,t)\in D^{1,2},\ \forall
k\in\mN, \; u_{n,k}\to u_n $ in $ L^2(\Omega_M), $ and it holds that
\begin{equation}
\sup_{k\in\mN} \|D_{.,.}u_{n,k}(x,t)\|^2_{L^2(\Om_M\times\cd
\times[0,\min\{T,\tau_M,\tau_{M_{\rm d}}\}))} = \underset{k\in\mN}{\text{sup}}\bigg[E\biggl{(}\int_0^{\min\{T,\tau_M,\tau_{M_{\rm d}}\}}\hspace*{-0.2cm}\int_{\cd}
|D_{y,s}u_{n,k}(x,t)|^2\, dyds\biggr{)}\bigg] < + \infty \,,
\end{equation}
thus, \cite[Lemma 1.2.3]{nual} asserts that $u_n(x,t)\in D^{1,2}$ and
\begin{align*}
D_{y,s}u_{n,k}(x,t)\rightarrow D_{y,s}u_{n}(x,t)\ \ \ \text{as}\ k\rightarrow+\infty ,
\end{align*}
in the weak topology of
$L^2\bigl{(}\Om_M\times\cd\times[0,\min\{T,\tau_M,\tau_{M_{\rm
d}}\})\bigr{)}$.
\\[0.5cm]
$(ii)$ Taking Malliavin derivative in both sides of \eqref{coint}
for $\eta:=0$, and using the analogous calculus and arguments, we
obtain $\forall (x,t)\in\cd\times[0,\min\{T,\tau_M,\tau_{M_{\rm
d}}\})$ and for any $s\leq t$,
\begin{eqnarray*}
D_{y,s}u_n(x,t)&=&G(x,y,t-s)\sigma(y) \,+\,\int_s^t\int_{\cd}G_{\Tilde{y}}(x,\Tilde{y},t-\Tilde{s})u_{\Tilde{y}}(0,\Tilde{s})
\ \cg_{n}(\Tilde{y},\Tilde{s})D_{y,s}u_{n}(\Tilde{y},\Tilde{s})\
d\Tilde{y}d\Tilde{s}\\
& &+\int_s^t\int_{\cd}G_{\Tilde{y}}(x,\Tilde{y},t-\Tilde{s})
D_{y,s}(u_{\Tilde{y}}(0,\Tilde{s}))\Tilde{y}T_n(\Tilde{y}^{-1} u_n(\Tilde{y},\Tilde{s}))\,
d\Tilde{y}\,d\Tilde{s},
\end{eqnarray*}
(i.e.~\eqref{41} is satisfied) while $\forall s>t$ it holds $D_{y,s}u_n(x,t)=0$ and $\cg_{n}$ is a random variable that satisfies \eqref{42}.

It remains to show the uniqueness of solution of \eqref{41}. If $\Tilde{D}_{y,s}u_{n}(x,t)$ is another solution of \eqref{41}, then since by definition of $\tau_{M_{\rm d}}$ the term $D_{y,s}(u_{\Tilde{y}}(0,\Tilde{s}))$ exists uniquely when $\Tilde{s}\in[0,\min\{T,\tau_M,\tau_{M_{\rm d}}\})$, through linearity and application of the same arguments, we get the analogous results.

More specifically, let $t\in [0,\min\{T,\tau_M,\tau_{M_{\rm d}}\})$ and defining
\begin{align*}
B_n(t) := E\biggl(\int_0^t\int_{\cd}\|D_{y,s}u_n(\cdot,t)-\Tilde{D}_{y,s}u_n(\cdot,t)\|^p_{L^p(\cd)}\,dy\,ds\biggr) ,
\end{align*}
we can analogously derive, as the first and last term of
\eqref{41} are common for $D_{y,s}$ and $\Tilde{D}_{y,s}$ and by
subtraction vanish, that
\begin{align*}
B_n(t)\leq c\int_0^t |t-\Tilde{s}|^{-\beta p} \, B_n(\Tilde{s})\, d\Tilde{s} ,
\end{align*}
and thus, by Gr\"onwall's lemma, we get $B_n(t)=0$ for any $t$, i.e.,

\begin{align*}
E\biggl(\int_0^t\int_{\cd}\|D_{y,s}u_n(\cdot,t)-\Tilde{D}_{y,s}u_n(\cdot,t)\|^p_{L^p(\cd)}\,dy\,ds\biggr)=0,\quad \forall t\in[0,\min\{T,\tau_M,\tau_{M_{\rm d}}\}) ,
\end{align*}
which yields the desired uniqueness of solution of \eqref{41}.
\\[0.5cm]
$(iii)$ Since $D_{.,.}u_{n,k}(x,t)\rightarrow D_{.,.}u_n(x,t)$
as $k\rightarrow+\infty$ weakly in
$L^2\bigl{(}\Om_M\times\cd\times[0,\min\{T,\tau_M,\tau_{M_{\rm d}}\})\bigr{)}$ for any
$(x,t)\in\cd\times[0,\min\{T,\tau_M,\tau_{M_{\rm d}}\})$, we have

\allowdisplaybreaks
\begin{align*}
E\biggl{(}\int_0^{\min\{T,\tau_M,\tau_{M_{\rm d}}\}}\int_{\cd}|D_{y,s}u_n(x,t)|^2\,dy\,ds\biggr{)}&= \|D_{.,.}u_n(x,t)\|^2_{L^2(\Om_M\times\cd\times[0,\min\{T,\tau_M,\tau_{M_{\rm d}}\}))}
\\[0.2cm]
&\leq \liminf_{k\rightarrow+\infty} \big(\|D_{.,.}u_{n,k}(x,t)\|^2_{L^2(\Om_M\times\cd\times[0,\min\{T,\tau_M,\tau_{M_{\rm d}}\}))}\big)
\\[0.2cm]
&=\lim_{k\rightarrow+\infty} \big(\underset{m\geq k}{\text{inf}}\|D_{.,.}u_{n,m}(x,t)\|^2_{L^2(\Om_M\times\cd\times[0,\min\{T,\tau_M,\tau_{M_{\rm d}}\}))}\big)
\\[0.2cm]
&\leq \underset{k\rightarrow+\infty}{\text{lim}}\big(\underset{m\geq k}{\text{sup}}\|D_{.,.}u_{n,m}(x,t)\|^2_{L^2
(\Om_M\times\cd\times[0,\min\{T,\tau_M,\tau_{M_{\rm d}}\}))}\big)
\\[0.2cm]
&\leq \sup_{k\in\mN} \|D_{.,.}u_{n,k}(x,t)\|^2_{L^2(\Om_M\times\cd\times[0,\min\{T,\tau_M,\tau_{M_{\rm d}}\}))}\leq c(n),
\end{align*}
and thus

\begin{align*}
&\|Du_{n}\|^2_{L^{2}(\Om_M\times(\cd\times[0,\min\{T,\tau_M,\tau_{M_{\rm d}}\}))^{2})}
:=
E\biggl{(}\int_0^{\min\{T,\tau_M,\tau_{M_{\rm d}}\}}\int_{\mathcal{D}}\int_0^{\min\{T,\tau_M,\tau_{M_{\rm d}}\}}\int_{\cd}\big{|}D_{y,s}u_n(x,t)\big{|}^{2}\, dydsdxdt\biggr{)}
\\
&= \int_0^{\min\{T,\tau_M,\tau_{M_{\rm d}}\}}\int_{\cd}E\biggl{(}\int_0^{\min\{T,\tau_M,\tau_{M_{\rm d}}\}}
\int_{\cd}\big{|}D_{y,s}u_{n}(x,t)\big{|}^{2}\,dy\,ds\biggr{)}dx\,dt < +\infty.
\end{align*}
Also, we have
\begin{align*}
\|u_n\|^2_{L^2(\Om_M\times\cd\times[0,\min\{T,\tau_M,\tau_{M_{\rm d}}\}))}&:=E\biggl{(}\int_0^{\min\{T,\tau_M,\tau_{M_{\rm d}}\}}
\int_{\cd}|u_n(x,t)|^2\,dx\,dt\biggr{)}\\[0.2cm] &\leq \int_0^{\min\{T,\tau_M,\tau_{M_{\rm d}}\}}E\biggl{(}\int_{\cd}|u_n(x,t)|^{2}\,dx\biggr{)}dt
\\[0.2cm]
&\leq \lambda \int_0^{\min\{T,\tau_M,\tau_{M_{\rm d}}\}}E\bigl{(}\|u_n(\cdot,t)\|^2_{L^{\infty}(\cd)}\bigr{)}dt
\\[0.2cm]
&\leq c \underset{t\in[0,\min\{T,\tau_M,\tau_{M_{\rm d}}\})}{\text{sup}}E\bigl{(}\|u_n(\cdot,t)\|^2_{L^{\infty}(\cd)}\bigr{)}<+\infty ,
\end{align*}
by \eqref{bound}. Therefore,
\begin{multline*}
\|u_n\|_{L^{1,2}(\cd\times[0,\min\{T,\tau_M,\tau_{M_{\rm d}}\}))}
\\
:=\left(\|u_{n}\|^{2}_{L^{2}(\Om_M\times\cd\times[0,\min\{T,\tau_M,\tau_{M_{\rm d}}\}))}
+\|Du_{n}\|^2_{L^2(\Om_M\times(\cd\times[0,\min\{T,\tau_M,\tau_{M_{\rm d}}\}))^2)}\right)^{1/2}
\, < +\infty, \qquad
\end{multline*}
which yields $u_n\in L^{1,2}\bigl{(}\cd\times[0,\min\{T,\tau_M,\tau_{M_{\rm
d}}\})\bigr{)}$.
\end{proof}

This section's main theorem is a direct consequence of the
previous arguments.

\begin{theorem}
The solution $u$ of (\ref{int}) for $\eta:=0$ belongs to $L^{1,2}_{\text{loc}}\bigl(\cd\times[0,\min\{T,\tau_M,\tau_{M_{\rm d}}\})\bigr)
\subseteq D^{1,2}_{\text{loc}}\bigl(\cd\times[0,\min\{T,\tau_M,\tau_{M_{\rm d}}\})\bigr)$.
\end{theorem}
\begin{proof}
The proof follows since we constructed a localization of $u$, by $(\Omega_M^n,u_n),\ n\in\mN$ with $u_n\in L^{1,2}\subseteq D^{1,2}$.
\end{proof}

\subsection{Existence of density}
For the solution $u$ of \eqref{int} for $\eta=0, $ we will show that
\begin{equation}\label{NualThm213}
P\left(\int_0^{\min\{T,\tau_M,\tau_{M_{\rm d}}\}}\int_{\cd}|D_{y,s}u(x,t)|^2 \,dy\,ds>0\right) =1 ,
\end{equation}
which, by \cite[Theorem 2.1.3]{nual}, implies that the law of $ u(x,t) $ is absolutely continuous with respect to the Lebesgue measure on $\mathbb R. $ By a localization argument (see \cite[Remark 3.1]{JDE18}), for proving  \eqref{NualThm213}, it is enough to show 
$$ P\left(\int_0^{\min\{T,\tau_M,\tau_{M_{\rm d}}\}}\int_{\cd}|D_{y,s}u_n(x,t)|^2 \,dy\,ds>0\right) =1 ,$$for the solution $ u_n$ of \eqref{coint}. To this end, we first prove two very important estimates, as in \cite{web,JDE18}.

We keep the definitions of $\Omega_M$ and $\tau_M$, and $\tau_{M_d}$ of the previous sections.

\begin{proposition}
Let a given deterministic $T>0 $ and $ 2 \leq p < 3. $ For any $b\in(0,\min\{T,\tau_M,\tau_{M_{\rm d}}\})$ and any $\epsilon\in(0,b]$, there exists a constant $ c > 0 $ independent of $\epsilon, $ such that
\begin{equation}\label{51}
\sup_{t\in[b-\epsilon,b]}
E\biggl{(}\int_{b-\epsilon}^b\sup_{x\in\mathcal{D}} \Big{[}\int_{\mathcal{D}}|D_{y,s}u_n(x,t)|^p dy\Big{]} ds\biggr{)}\leq c\,\epsilon^{\frac{3-p}{2}},
\end{equation}
and, for any $ \delta\in (0,1), $ there exists a constant $ c > 0 $ independent of $\epsilon, $ such that
\begin{equation}\label{52}
\sup_{t\in [\epsilon,\min\{T,\tau_M,\tau_{M_{\rm
d}}\})} E\biggl{(}\int_{t-\epsilon}^t \sup_{x\in\mathcal{D}}
\Big{[}\int_{\mathcal{D}}|D_{y,s}u_n(x,t)-G(x,y,t-s)\sigma(y)|^p\,dy\Big{]}
ds\biggr{)}\\
\leq c\,\epsilon^{1-\delta}.
\end{equation}
\end{proposition}
\begin{proof}We denote the last term of \eqref{41} by
\begin{equation}\label{lt}
A(x,y,s,t):=\int_s^t\int_{\cd}G_{\Tilde{y}}(x,\Tilde{y},t-\Tilde{s})
D_{y,s}(u_{\Tilde{y}}(0,\Tilde{s}))\Tilde{y}T_n(\Tilde{y}^{-1}u_{n}(\Tilde{y},\Tilde{s}))\,d\Tilde{y}\,d\Tilde{s}.
\end{equation}

Taking absolute value and then raising to the $p$
power in \eqref{41}, we get, for some constant $c_p>0,$
\begin{eqnarray}
|D_{y,s}u_{n}(x,t)|^{p}&\leq&
c_p |G(x,y,t-s)\sigma(y)|^{p} 
\nonumber\\[0.5em]
& &+ c_p  \bigg{|}\int_s^t
\int_{\cd}G_{\Tilde{y}}(x,\Tilde{y},t-\Tilde{s})u_{\Tilde{y}}(0,\Tilde{s})\cg_{n}(\Tilde{y},\Tilde{s})D_{y,s}u_n(\Tilde{y},\Tilde{s})\,d\Tilde{y}\,d\Tilde{s}\bigg{|}^p
+ c_p |A|^p\;.
\label{ttt3}
\end{eqnarray}

Let $b\in[0,\min\{T,\tau_M,\tau_{M_{\rm d}}\})$ and $\epsilon\in(0,b]$. From now on we work for times $t\in[b-\epsilon,b]$. In what follows, $c $ denotes a positive constant independent of $x, t $ and $ \epsilon, $ and $c$ may change from one inequality to another.

In \eqref{ttt3}, we integrate in $y\in\mathcal{D}$, then take supremum in $x$ and integrate for $s\in[b-\epsilon,t]$, and then take expectation, to arrive at
\begin{flalign}\label{totre}
E\biggl{(}\int_{b-\epsilon}^{t}\sup\limits_{x\in\mathcal{D}} &\int_{\mathcal{D}}|D_{y,s}u_{n}(x,t)|^{p}dy \, ds\biggr{)}
\leq c_p\int_{b-\epsilon}^{t}\displaystyle{\sup_{x\in\mathcal{D}}}\int_{\mathcal{D}}|G(x,y,t-s)\sigma(y)|^p\,dy\,ds 
\nonumber\\
&+c_pE\biggl{(}\int_{b-\epsilon}^t\displaystyle{\sup_{x\in\mathcal{D}}}\int_{\mathcal{D}}\bigg{|}\int_{s}^t\int_{\cd}G_{\Tilde{y}}(x,\Tilde{y},t-\Tilde{s})\
[u_{\Tilde{y}}(0,\Tilde{s})\cg_{n}(\Tilde{y},\Tilde{s})D_{y,s}u_{n}(\Tilde{y},\Tilde{s})]\,d\Tilde{y}\,d\Tilde{s}\bigg{|}^pdy\,ds\biggr{)}
\nonumber\\
& + c_p E\biggl{(}\int_{b-\epsilon}^t\displaystyle{\sup_{x\in\mathcal{D}}}\int_{\mathcal{D}} |A|^pdy\, ds\biggr{)}
\nonumber\\
=: &\; \mathcal B_1+\mathcal B_2+\mathcal B_3.
\end{flalign}

Integrating in $y$ and using \eqref{ggg} and \eqref{kostasgaus}, we have
\begin{eqnarray}
\int_{\mathcal{D}}|G(x,y,t-s)|^p\,|\sigma(y)|^p\,dy&\leq &
\,c\, |t-s|^{-\frac{p}{2}}\int_{\mathcal{D}}\exp\Big{(}\!-c\,\frac{|x-y|^{2}}{|t-s|}\Big{)}\,|\sigma(y)|^p\,dy\nonumber\\
&\leq & 
\,c\, |t-s|^{-\frac{p}{2}}\int_{\mathcal{D}}\exp\Big{(}\!-c\frac{|x-y|^{2}}{|t-s|}\Big{)}dy\nonumber\\
&\leq&
\,c\,|t-s|^{-\frac{p}{2}}|t-s|^{\frac{1}{2}}=c|t-s|^{\frac{1-p}{2}}.\label{ffff}
\end{eqnarray}
The bound \eqref{ffff} is integrable in time if $p<3$. Using \eqref{ffff}, we get
\begin{eqnarray}\label{1term}
\mathcal{B}_1&:=&c_p\int_{b-\epsilon}^t\displaystyle{\sup_{x\in\mathcal{D}}}\int_{\mathcal{D}}|G(x,y,t-s)\sigma(y)|^pdy\,ds
\nonumber\\
&\leq& c\int_{b-\epsilon}^{t}|t-s|^{\frac{1-p}{2}} ds
\nonumber\\
&=&c\,(t-b+\epsilon)^{\frac{3-p}{2}}.
\end{eqnarray}
Regarding the term $\mathcal{B}_2$ in \eqref{totre}, we have
\begin{eqnarray}
\mathcal{B}_2 &:= & c_p E\left(\int_{b-\epsilon}^{t} \sup\limits_{x\in\mathcal{D}} \int_{\mathcal{D}}
\bigg{|}\int_s^t\int_{\cd}G_{\Tilde{y}}(x,\Tilde{y},t-\Tilde{s})\
[u_{\Tilde{y}}(0,\Tilde{s})\cg_{n}(\Tilde{y},\Tilde{s})D_{y,s}u_{n}(\Tilde{y},\Tilde{s})]\,d\Tilde{y}\,d\Tilde{s}\bigg{|}^pdy\,ds\right)
\nonumber\\
&\leq&cE\biggl{(}\int_{b-\epsilon}^t\sup\limits_{x\in\mathcal{D}} \int_{\mathcal{D}}
\bigg{|}\int_s^t \int_{\cd}|G_{\Tilde{y}}(x,\Tilde{y},t-\Tilde{s})|\
|D_{y,s}u_{n}(\Tilde{y},\Tilde{s})|\
d\Tilde{y}d\Tilde{s}\bigg{|}^{p}dy\,ds\biggr{)}
\nonumber\\
&\leq&cE\biggl{(}\int_{b-\epsilon}^t \sup_{x\in\mathcal{D}} \int_{\mathcal{D}}
\bigg{|}\int_s^t\dfrac{1}{|t-\Tilde{s}|} \; \|D_{y,s}u_n(\cdot,\Tilde{s})\|_{L^p(\mathcal{D})} \left(\int_{\cd}\text{exp}\biggl{(}-\frac{c |x-\Tilde{y}|^2}{|t-\tilde{s}|}\biggr)\, d\Tilde{y}\right)^{1/p'}\!d\Tilde{s}\bigg{|}^p dy\,ds\biggr{)}
\nonumber\\
&\leq&cE\biggl{(}\int_{b-\epsilon}^{t}\int_{\mathcal{D}}
\bigg{|}\int_{s}^{t}|t-\tilde{s}|^{\frac{1}{2p'}-1} \|D_{y,s}u_n(\cdot,\Tilde{s})\|_{L^p(\mathcal{D})}\,d\Tilde{s}\bigg{|}^pdy\,ds\biggr{)}
\nonumber\\
&\leq& c \, E\biggl{(}\int_{b-\epsilon}^t \int_{\mathcal{D}} \left(\int_s^t |t-\tilde{s}|^{-\beta p}
 \|D_{y,s}u_n(\cdot,\Tilde{s})\|_{L^p(\mathcal{D})}^p\,d\Tilde{s}\right) dy \,ds\biggr{)}
\nonumber\\
&=&c \, E\biggl{(}\int_{b-\epsilon}^t \int_{b-\epsilon}^{\tilde{s}} \int_{\mathcal{D}} |t-\tilde{s}|^{-\beta p} \, \|D_{y,s}u_n(\cdot,\Tilde{s})\|_{L^p(\mathcal{D})}^p\,dy\,ds\,d\Tilde{s}\biggr{)}
\nonumber\\
&=&c \, \int_{b-\epsilon}^t E\biggl{(}\int_{b-\epsilon}^{\tilde{s}} |t-\tilde{s}|^{-\beta p} \,\int_{\mathcal D} \int_{\mathcal{D}} |D_{y,s}u_n(\tilde{y},\Tilde{s})|^p\,d\tilde{y}\, dy\,ds\biggr{)} \,d\Tilde{s}
\nonumber\\
&=&c \, \int_{b-\epsilon}^t E\biggl{(}\int_{b-\epsilon}^{\tilde{s}} |t-\tilde{s}|^{-\beta p} \, \int_{\mathcal{D}}\left(\int_{\mathcal D} |D_{y,s}u_{n}(\tilde{y},\Tilde{s})|^p\, dy\right) d\tilde{y} \,ds\biggr{)}\,d\Tilde{s}
\nonumber\\
&\leq & c 
\,\int_{b-\epsilon}^{t}|t-\tilde{s}|^{-\beta p} E\left(\int_{b-\epsilon}^{\tilde{s}} \sup\limits_{\tilde{y}\in\mathcal{D}} \int_{\mathcal{D}}
|D_{y,s}u_n(\Tilde{y},\Tilde{s})|^p\,dy\,ds\right)d\Tilde{s} ,
\label{sos}
\end{eqnarray}
where we used for the first inequality the uniform
upper-boundedness of $|\cg_{n}|$ by the deterministic constant
$C_L$, and also the upper-boundedness of
$|u_{\Tilde{y}}(0,\Tilde{s})|$ by $M$, the estimate \eqref{gen} and H\"older inequality with $1/p+1/p'=1$ for the second inequality, the \eqref{kostasgaus} for the third and H\"older for the fourth inequality for $\beta>(3-p)/(2p)$. In the last equalities we changed the order of integration, and in the last inequality we used that $\|\cdot\|_{L^p(\mathcal{D})}\leq \lambda^{1/p} \|\cdot\|_{L^\infty(\mathcal D)}$. So we get
\begin{equation}\label{2term}
\mathcal{B}_2 \leq c \int_{b-\epsilon}^{t}|t-\tilde{s}|^{-\beta p} E\left(\int_{b-\epsilon}^{\tilde{s}} \sup\limits_{\tilde{y}\in\mathcal{D}} \int_{\mathcal{D}}
|D_{y,s}u_{n}(\Tilde{y},\Tilde{s})|^p\,dy\,ds\right)d\Tilde{s}.
\end{equation}

Regarding the term $\mathcal{B}_3$ in \eqref{totre}, we have
\begin{eqnarray}
\mathcal B_3 &:=& c_p E\left(\int_{b-\epsilon}^t \sup_{x\in\mathcal{D}}\int_{\mathcal{D}}|A|^p\,dy\,ds\right)
\nonumber\\
&=&
c_p E\biggl{(}
\int_{b-\epsilon}^t\displaystyle{\sup_{x\in\mathcal{D}}}\int_{\mathcal{D}}\Big{|}\int_s^t\int_{\cd}G_{\Tilde{y}}(x,\Tilde{y},t-\Tilde{s})
D_{y,s}(u_{\Tilde{y}}(0,\Tilde{s}))\Tilde{y}T_n(\Tilde{y}^{-1}u_{n}(\Tilde{y},\Tilde{s}))\,d\Tilde{y}\,d\Tilde{s}\Big{|}^p\,dy\,ds \biggr{)}
\nonumber\\
&\leq& c \, E\biggl{(}\int_{b-\epsilon}^{t}\displaystyle{\sup_{x\in\mathcal{D}}}\int_{\mathcal{D}} \Big{|}\int_{s}^{t}\int_{\cd}|G_{\Tilde{y}}(x,\Tilde{y},t-\Tilde{s})|
\,d\Tilde{y}\,d\Tilde{s}\Big{|}^pdy\,ds\biggr{)}
\nonumber \\
&\leq& \,c\, \int_{b-\epsilon}^t|t-s|^{p/2}ds
\nonumber \\
&= & c\, (t-b+\epsilon)^{\frac{p}{2}+1} , \label{3term1}
\end{eqnarray}
where we used that $u_n=u$ in $\Omega_M^n$ and so it holds that
$$\displaystyle{\sup_{\Tilde{y}\in\mathcal{D}}}|\Tilde{y}^{-1}u_n(\Tilde{y},\Tilde{s})|<
n , $$
therefore
$$
T_n(\Tilde{y}^{-1}u_n(\Tilde{y},\Tilde{s}))=\Tilde{y}^{-1}u_n(\Tilde{y},\Tilde{s})<n,$$
which gives that
$$\Tilde{y}\displaystyle{\sup_{\Tilde{y}\in\mathcal{D}}}|T_n(\Tilde{y}^{-1}u_{n}(\Tilde{y},\Tilde{s}))|<\lambda n,$$
while we also used the boundedness of
$D_{y,s}(u_{\Tilde{y}}(0,\Tilde{s}))$ by $M_d$, and the estimate
\eqref{gy}.

Using in \eqref{totre} the estimates \eqref{1term},
\eqref{2term}, \eqref{3term1}, we get, for $p<3 $
and $c=c(n)$ independent of $t, $ that

\begin{flalign}
 E\left(\int_{b-\epsilon}^{t} \sup\limits_{x\in\mathcal{D}} \int_{\mathcal{D}}|D_{y,s}u_{n}(x,t)|^{p}dy ds\right)
& \leq \, c(n)(t-b+\epsilon)^{\frac{3-p}{2}} + c(n) (t-b+\epsilon)^{\frac{p}{2}+1} 
\nonumber\\[0.3em]
 + &c(n)\int_{b-\epsilon}^{t}|t-\tilde{s}|^{-\beta p}E\left(\int_{b-\epsilon}^{\tilde{s}} \sup\limits_{\tilde{y}\in\mathcal{D}} \int_{\mathcal{D}}
|D_{y,s}u_{n}(\Tilde{y},\Tilde{s})|^p\ dy ds\right)d\Tilde{s}\;.
\label{totre2n}
\end{flalign}
Setting 
$$L_n(t):= E\left(\int_{b-\epsilon}^t \sup\limits_{x\in\mathcal{D}} \int_{\mathcal{D}}|D_{y,s}u_{n}(x,t)|^{p}dy ds\right)\;,$$
the inequality \eqref{totre2n} becomes 

\begin{equation}\label{totre3f} L_n(t) \leq c (t-b+\epsilon)^{\frac{3-p}{2}} + c (t-b+\epsilon)^{\frac{p}{2}+1} + c\int_{b-\epsilon}^t|t-\tilde{s}|^{-\beta p} \,L_n(\tilde{s})\,d\tilde{s}.\end{equation}
By \eqref{totre3f}, we get 

\begin{eqnarray}
\int_{b-\epsilon}^t|t-\tilde{s}|^{-\beta p}\, L_n(\tilde{s})\,d\tilde{s}
&\leq&
c \int_{b-\epsilon}^t |t-\Tilde{s}|^{-\beta p}\, (\tilde{s}-b+\epsilon)^{\frac{3-p}{2}}\,d\tilde{s} 
+ c \int_{b-\epsilon}^t |t-\Tilde{s}|^{-\beta p} \, (\tilde{s}-b+\epsilon)^{\frac{p}{2}+1}\,d\tilde{s} 
\nonumber\\
& & + c \int_{b-\epsilon}^t |t-\Tilde{s}|^{-\beta p} \int_{b-\epsilon}^{\tilde{s}} |\tilde{s}-\tau|^{-\beta p} \,  L_n(\tau) \, d\tau \, d\tilde{s}
\nonumber\\
&\leq&
c (t-b+\epsilon)^{\frac{3-p}{2}}\, \int_{b-\epsilon}^t |t-\Tilde{s}|^{-\beta p}\, d\tilde{s} 
+ c  (t-b+\epsilon)^{\frac{p}{2}+1}\,\int_{b-\epsilon}^t |t-\Tilde{s}|^{-\beta p} \,d\tilde{s} 
\nonumber\\
& & + c \int_{b-\epsilon}^t |t-\Tilde{s}|^{-\beta p} \int_{b-\epsilon}^{\tilde{s}}  |\tilde{s}-\tau|^{-\beta p} \, L_n(\tau) \, d\tau \, d\tilde{s}
\nonumber\\
&=&
\frac{c}{1-\beta p} \, (t-b+\epsilon)^{\frac{3-p}{2}+1-\beta p}
\,+\,
\frac{c}{1-\beta p} \, (t-b+\epsilon)^{\frac{p}{2}+1+1-\beta p}
\nonumber\\
& & + c \int_{b-\epsilon}^t \Big[ \int_{\tau}^t |t-\Tilde{s}|^{-\beta p} \, |\tilde{s}-\tau|^{-\beta p} \, d\tilde{s} \Big]\, L_n(\tau) \, d\tau
\nonumber\\
&=&
c\, (t-b+\epsilon)^{\frac{5-p}{2}-\beta p}
\,+\,
c \, (t-b+\epsilon)^{\frac{p+4}{2}-\beta p}
\, +\, c \int_{b-\epsilon}^t L_n(\tau) \, d\tau \,,\label{gineqp2}
\end{eqnarray}
where we used that (see \eqref{hineq})
$$
\int_\tau^t|t-\Tilde{s}|^{-\beta p} \, |\tilde{s}-\tau|^{-\beta p} \,  \,d\tilde{s} \leq C < +\infty.
$$
Using \eqref{gineqp2} in the last term of \eqref{totre3f}, we obtain

\begin{equation}
L_n(t) \leq c (t-b+\epsilon)^{\frac{3-p}{2}} + c (t-b+\epsilon)^{\frac{p}{2}+1}\,+ c\, (t-b+\epsilon)^{\frac{5-p}{2}-\beta p}
\,+\,c \, (t-b+\epsilon)^{\frac{p+4}{2}-\beta p}
\, +\, c \int_{b-\epsilon}^t L_n(\tau) \, d\tau\;.
\end{equation}
By Gr\"onwall's lemma, we conclude that, for $t\leq b < \min\{T,\tau_M,\tau_{M_{\rm d}}\}, $

$$L_n(t) \leq c (t-b+\epsilon)^{\frac{3-p}{2}} + c (t-b+\epsilon)^{\frac{p}{2}+1}\,+ c\, (t-b+\epsilon)^{\frac{5-p}{2}-\beta p}
\,+\,
c \, (t-b+\epsilon)^{\frac{p+4}{2}-\beta p} \,,$$
that is

\begin{eqnarray}\label{keq1}
E\left(\int_{b-\epsilon}^t \sup\limits_{x\in\mathcal{D}} \int_{\mathcal{D}}|D_{y,s}u_{n}(x,t)|^p\,dy\, ds\right)
&\leq& c \,(t-b+\epsilon)^{\frac{3-p}{2}} \,+\, c\, (t-b+\epsilon)^{\frac{p}{2}+1}
\nonumber\\
& & + \,c \,(t-b+\epsilon)^{\frac{5-p}{2}-\beta p}\,+\,
c \, (t-b+\epsilon)^{\frac{p+4}{2}-\beta p}.
\end{eqnarray}
Using that $ D_{y,s}\bigl{(}u_n(x,t)\bigr{)}=0 $ for $s>t$ and any $x\in\cd$, and \eqref{keq1}, we obtain, for $ b-\epsilon\leq t\leq b,$

\begin{eqnarray}\label{keq2}
E\left(\int_{b-\epsilon}^b \sup\limits_{x\in\mathcal{D}} \int_{\mathcal{D}}|D_{y,s}u_{n}(x,t)|^p\,dy\,ds\right)
&=&
E\left(\int_{b-\epsilon}^t \sup\limits_{x\in\mathcal{D}} \int_{\mathcal{D}}|D_{y,s}u_{n}(x,t)|^p\,dy\,ds\right)
\nonumber\\
&\leq&
 c \,(t-b+\epsilon)^{\frac{3-p}{2}} \,+\, c\, (t-b+\epsilon)^{\frac{p}{2}+1}
\nonumber\\
& &+\, c\, (t-b+\epsilon)^{\frac{5-p}{2}-\beta p}
\,+\,c \, (t-b+\epsilon)^{\frac{p+4}{2}-\beta p}
\nonumber\\
&\leq&
 c \,\epsilon^{\frac{3-p}{2}} \,+\, c\, \epsilon^{\frac{p}{2}+1}
\,+\, c\, \epsilon^{\frac{5-p}{2}-\beta p}
\,+\,c \, \epsilon^{\frac{p+4}{2}-\beta p}\nonumber\\
&\leq&
 c \,\epsilon^{\frac{3-p}{2}} \,,\label{keq3}
\end{eqnarray}
for $ \beta \in ((3-p)/(2p),1/(2p)). $ Taking the supremum in \eqref{keq3} over $t\in[b-\epsilon,b], $ we get \eqref{51}. 

Furthermore, using \eqref{41}, \eqref{2term}, \eqref{3term1}  with $b=t$ and \eqref{51}, we obtain
\begin{eqnarray*}\label{keq4}
& &E\biggl{(}\int_{t-\epsilon}^t \sup_{x\in\mathcal{D}}
\Big{[}\int_{\mathcal{D}}|D_{y,s}u_n(x,t)-G(x,y,t-s)\sigma(y)|^p\,dy\Big{]}
ds\biggr{)}
\\[0.5em]
& &\qquad\leq  2^{p-1}E\left(\int_{t-\epsilon}^t \sup\limits_{x\in\mathcal{D}} \int_{\mathcal{D}}
\bigg{|}\int_s^t\int_{\cd}G_{\Tilde{y}}(x,\Tilde{y},t-\Tilde{s})\
[u_{\Tilde{y}}(0,\Tilde{s})\cg_{n}(\Tilde{y},\Tilde{s})D_{y,s}u_n(\Tilde{y},\Tilde{s})]\,d\Tilde{y}\,d\Tilde{s}\bigg{|}^p dy\,ds\right)
\\
& &\qquad\qquad+\;2^{p-1} E\left(\int_{t-\epsilon}^t \sup_{x\in\mathcal{D}}\int_{\mathcal{D}}|A|^p\,dy\,ds\right)
\\[0.5em]
& &\qquad \leq 
c\,\epsilon^{1+\frac{3-p}{2}-\beta p} \,+\, c\,\epsilon^{\frac{p}{2}+1}
\\[0.5em]
& &\qquad\leq  c\,\epsilon^{1+\frac{3-p}{2}-\beta p} ,
\end{eqnarray*}
which proves the estimate \eqref{52} with $\delta:=\beta p-(3-p)/2.$
\end{proof}

\begin{proposition}Consider a given deterministic $T>0,  $ and $ x\in \mathcal D $ with $ \sigma(x) \neq 0. $ Then, for any $t\in[0,\min\{T,\tau_M,\tau_{M_{\rm d}}\})$, we have that
\begin{equation}\label{dc1}
P\big(\om\in\Om_M^n:\;\|D_{\cdot,\cdot}(u_n(\om;x,t))\|_{L^2(\cd\times[0,\min\{T,\tau_M,\tau_{M_{\rm d}}\}))}>0\big)=1,
\end{equation}
where $u_n$ is the unique solution of \eqref{coint}.
\end{proposition}
\begin{proof}
We keep $t$ less than $T$ deterministic and the stopping times
used in the previous Proposition. Recall \eqref{41} i.e.,
\begin{equation*}
\begin{split}
D_{y,s}u_n(x,t)=\, & G(x,y,t-s)\sigma(y)+\int_s^t\int_{\cd}G_{\Tilde{y}}(x,\Tilde{y},t-\Tilde{s})u_{\Tilde{y}}(0,\Tilde{s})
\ \cg_{n}(\Tilde{y},\Tilde{s})D_{y,s}u_{n}(\Tilde{y},\Tilde{s})\,d\Tilde{y}\,d\Tilde{s}\\
&+\int_s^t\int_{\cd}G_{\Tilde{y}}(x,\Tilde{y},t-\Tilde{s})
D_{y,s}(u_{\Tilde{y}}(0,\Tilde{s}))\Tilde{y}T_n(\Tilde{y}^{-1}u_n(\Tilde{y},\Tilde{s}))\,d\Tilde{y}\,d\Tilde{s}.
\end{split}
\end{equation*}
In the above, taking absolute value and then raising to the second power, we get
\begin{flalign*}
|D_{y,s}u_n(x,t)|^2=\bigg{|}G(x,y,t-s)\sigma(y)&+\int_s^t\int_{\cd}G_{\Tilde{y}}(x,\Tilde{y},t-\Tilde{s})u_{\Tilde{y}}(0,\Tilde{s})
\ \cg_{n}(\Tilde{y},\Tilde{s})D_{y,s}u_n(\Tilde{y},\Tilde{s})\,d\Tilde{y}\,d\Tilde{s}
\nonumber\\
&+\int_s^t\int_{\cd}G_{\Tilde{y}}(x,\Tilde{y},t-\Tilde{s})
D_{y,s}(u_{\Tilde{y}}(0,\Tilde{s}))\Tilde{y} T_n(\Tilde{y}^{-1}u_n(\Tilde{y},\Tilde{s}))\,d\Tilde{y}\,d\Tilde{s}\biggl{|}^2
\nonumber\\
\geq
\frac12A^2-B^2,
\end{flalign*}
where
$$A(x,y,t,s) := G(x,y,t-s)\sigma(y) ,$$ and
\begin{flalign*}
B(x,y,t,s) :=&
\int_s^t\int_{\cd}G_{\Tilde{y}}(x,\Tilde{y},t-\Tilde{s})u_{\Tilde{y}}(0,\Tilde{s})\
\cg_{n}(\Tilde{y},\Tilde{s})D_{y,s}u_{n}(\Tilde{y},\Tilde{s})\,d\Tilde{y}\,d\Tilde{s}\\
&+\int_s^t\int_{\cd}G_{\Tilde{y}}(x,\Tilde{y},t-\Tilde{s})
D_{y,s}(u_{\Tilde{y}}(0,\Tilde{s}))\Tilde{y}T_n(\Tilde{y}^{-1}u_n(\Tilde{y},\Tilde{s}))\,d\Tilde{y}\,d\Tilde{s}.
\end{flalign*}
Then, we integrate to obtain

\begin{equation}\label{ko1}
\int_{t-\epsilon}^t\int_{\mathcal{D}}|D_{y,s}u_n(x,t)|^2\,dy\,ds
\geq
\frac12\int_{t-\epsilon}^t \int_{\mathcal{D}}A^2\, dy\,ds
\,-\,
\int_{t-\epsilon}^t\int_{\mathcal{D}}B^2\,dy\,ds .
\end{equation}
For the first term on the right-hand side of \eqref{ko1}, we have (cf. \cite[Lemma B.2.3]{DS}), for small $\epsilon,$
\begin{eqnarray}\label{kwstas1}
\int_{t-\epsilon}^t \int_{\mathcal{D}}A^2\, dy\,ds
&=&
\int_0^\epsilon\int_{\cd}G^2(x,y,\tau)\,\sigma^2(y)\,dy\,d\tau 
\nonumber
\\
&\geq&
\int_0^\epsilon\int_{x-\sqrt{\epsilon}}^{x+\sqrt{\epsilon}}G^2(x,y,\tau)\,\sigma^2(y)\,dy\,d\tau
\nonumber
\\
&\geq&
\frac{\sigma^2(x)}{2}\int_0^\epsilon\int_{x-\sqrt{\epsilon}}^{x+\sqrt{\epsilon}}G^2(x,y,\tau)\,dy\,d\tau
\nonumber
\\
& \geq & c_1\, \epsilon^{1/2}.
\end{eqnarray}
For the second term on the right-hand side of \eqref{ko1}, we use \eqref{41} and \eqref{52}, to obtain that
\begin{eqnarray}
E\biggl{(}\int_{t-\epsilon}^t\int_{\cd}B^2\,dy\,ds\biggr{)}&=&
E\biggl{(}\int_{t-\epsilon}^t\int_{\cd}\Big{|}D_{y,s}u_n(x,t)-G(x,y,t-s)\sigma(y)\Big{|}^2\,dy\,ds\biggr{)}
\nonumber\\
&\leq&
c_2\,\epsilon^{1-\delta} ,\label{kwstas4}
\end{eqnarray}
for any $\delta \in (0,1). $ By \eqref{ko1}-\eqref{kwstas4} and Markov inequality, we get, for $ 0<\delta<1/2,$

\begin{eqnarray*}
P\biggl{(}\int_0^{\min\{T,\tau_M,\tau_{M_{\rm d}}\}}\int_{\cd}|D_{y,s}u_n(x,t)|^2dyds>0\biggr{)}
&\geq&
P\biggl{(}\frac12\int_{t-\epsilon}^t\int_{\cd}A^2 dyds-
\int_{t-\epsilon}^t\int_{\cd}B^2 dyds>0\biggr{)}
\\
&\geq&
P\biggl{(}\int_{t-\epsilon}^t\int_{\cd}B^2\, dy\,ds<\frac{c_1}{2}\epsilon^{1/2}\biggr{)}
\\
&\geq&
1-c\,\epsilon^{-1/2}E\biggl{(}\int_{t-\epsilon}^t\int_{\cd}B^2\,dy\,ds\biggr{)}
\\
&\geq&
1\,-\,c\,\epsilon^{-1/2}\,\epsilon^{1-\delta} \rightarrow 1, \qquad \text{as}\ \epsilon\rightarrow 0 ,
\end{eqnarray*}
which yields \eqref{dc1}.
\end{proof}

\begin{theorem}\label{thmkostas}Let $u$ be the solution of \eqref{int}. Then, for any $(x,t)\in \cd \times (0,\min\{T,\tau_M,\tau_{M_{\rm d}}\}) $ with $\sigma(x)\neq 0, $ we have that
\begin{equation*}
P\big(\{\om\in\Om_M:\; \left\|D_{\cdot,\cdot}(u(\om;x,t))\right\|_{L^2(\cd\times[0,\min\{T,\tau_M,\tau_{M_{\rm d}}\}))}>0\}\big)=1,
\end{equation*}
and the law of $ u(x,t) $ is absolutely continuous with respect to the Lebesgue measure on $\mathbb R.$
\end{theorem}

\begin{proof}Using that $ \Om_M^1\subseteq\Om_M^2\subseteq\dots\subseteq\Om_M
$ so that $ \lim_{n\rightarrow+\infty} P(\Om_M^n) = P(\Om_M), $
the fact that  $D_{y,s}u(x,t,\omega)=D_{y,s}u_n(x,t,\omega)$ on $\Omega_M^n$ a.s., and \eqref{dc1}, we obtain that
\begin{flalign*}
P\big(\{\om\in\Om_M:\; \left\|D_{\cdot,\cdot}(u(\om;x,t))\right\|&_{L^2(\cd\times[0,\min\{T,\tau_M,\tau_{M_{\rm d}}\}))}>0\}\big)=
\\
= &\lim_{n\to\infty}P\big(\{\om\in\Om_M^n:\; \left\|D_{\cdot,\cdot}(u(\om;x,t))\right\|_{L^2(\cd\times[0,\min\{T,\tau_M,\tau_{M_{\rm d}}\}))}>0\}\big)
\\
= &\lim_{n\to\infty}P\big(\{\om\in\Om_M^n:\; \left\|D_{\cdot,\cdot}(u_n(\om;x,t))\right\|_{L^2(\cd\times[0,\min\{T,\tau_M,\tau_{M_{\rm d}}\}))}>0\}\big)
\\
= &1,
\end{flalign*}
therefore,
 \begin{equation}\label{kostas11}
P\Big(\int_0^{\min\{T,\tau_M,\tau_{M_{\rm d}}\}}\int_{\cd}|D_{y,s}u(x,t)|^2 \,dy\,ds>0\Big) =1.
\end{equation}
We also have $u(x,t)\in D^{1,2}_{\text{loc}}\subset D^{1,1}_{\text{loc}}. $ Therefore, \eqref{kostas11} and \cite[Theorem 2.1.3]{nual} imply that the law of $ u(x,t) $ is absolutely continuous with respect to the Lebesgue measure on $\mathbb R.$
\end{proof}

\vspace{0.5cm}

\begin{center}
{\large \bfseries Declarations:}
\end{center}

\paragraph*{\bfseries Acknowledgement} The research work is implemented in the framework of H.F.R.I call “Basic research Financing (Horizontal support of all Sciences)” under the National Recovery and Resilience Plan “Greece 2.0” funded by the European Union – NextGenerationEU. (H.F.R.I. Project Number: 14910).

\vspace{0.5cm}

\paragraph*{\bfseries Conflicts of interests/Competing interests} The authors have no conflicts of interest to declare that are relevant to the content of this article.

\vspace{0.5cm}

\paragraph*{\bfseries Data availability} Data sharing is not applicable to this article as no datasets were generated or analysed during the current study.

\bibliographystyle{plain}

\end{document}